\documentclass[a4paper, 11pt, english]{article}
\usepackage[utf8]{inputenc}
\usepackage[T1]{fontenc}
\usepackage{graphicx}
\usepackage{stmaryrd}
\usepackage[a4paper]{geometry}
\usepackage{amsmath,amsfonts,amssymb,amsthm,epsfig,epstopdf,url,array}
\usepackage{rotating}
\usepackage[colorlinks=true,citecolor=red,linkcolor=blue,pdfpagetransition=Blinds]{hyperref}
\usepackage{cleveref}
\usepackage{nameref}
\usepackage{enumitem}
\usepackage{comment}
\Crefname{paragraph}{Section}{Sections}
\usepackage{fancyhdr}
\usepackage{booktabs}
\usepackage{multirow}
\usepackage{siunitx}

\usepackage{fullpage}

\usepackage{tikz}
\usepackage{pgfplots}
\usepackage{subfig}

\pgfplotsset{surface/.style={ %
		view={20}{35},
               axis z line=middle,%
               axis x line=center,%
               axis y line=center,%
               xmax=0.11,%
               clip=false,
               xtick=\empty,
               extra x ticks={0,0.1},%
               extra x tick labels={0,$\qquad T=0.1$},%
               ymajorticks=false,
               zmajorticks=true,
               colormap/jet}
               }

\pgfplotsset{erreurs/.style={legend cell align=left,
    legend pos=outer north east,
    legend plot pos=left,
    legend style={cells={anchor=west},draw=none},
    xlabel=$h$,
    xmin=0.001,
    xmax=0.03}
    }

\tikzset{pente/.style={opacity=0.6}}

\pgfplotsset{shadow/.style={black,mark=square,mark size=3.0,mark options={solid,fill=none}}}
\pgfplotsset{fast/.style={black,mark=triangle*,mark size=3.0,mark options={solid,fill=none}}}

\pgfplotsset{cout/.style={purple,mark=diamond*,mark size=3.2,mark options={fill=purple}}}
\pgfplotsset{cible/.style={blue,mark=square*,mark size=2.5,mark options={fill=blue}}}
\pgfplotsset{cibleyT/.style={black,mark=*,mark size=2.5,mark options={fill=gray}}}
\pgfplotsset{CG/.style={black,mark=otimes*,mark size=2.5,mark options={fill=gray}}}
\pgfplotsset{solex/.style={black,mark=*,mark size=2.5,mark options={fill=gray}}}
\pgfplotsset{energie/.style={red,mark=*,mark size=2.5,mark options={fill=red}}}

\usetikzlibrary{matrix,external}

\newcommand{\ensemblenombre}[1]{\mathbb{#1}}

\newcommand{\R}{} 
\renewcommand{\R}{\ensemblenombre{R}}

\newcommand{\K}{\ensemblenombre{K}}

\newcommand\D{\displaystyle}

\newcommand{\norme}[1]{\left\lVert#1\right\rVert}

\theoremstyle{plain} 
\newtheorem{prop}{Proposition}[section] 
\newtheorem{theo}[prop]{Theorem}

\newtheorem{lem}[prop]{Lemma}

\theoremstyle{definition}

\newtheorem{rmk}[prop]{Remark}

\newtheorem{claim}[prop]{Claim}

\newtheorem{op}[prop]{Open question}

\def\dx{\,\textnormal{d}x}
\def\dt{\textnormal{d}t}
\def\d{\,\textnormal{d}}

\def\Xint#1{\mathchoice
{\XXint\displaystyle\textstyle{#1}}%
{\XXint\textstyle\scriptstyle{#1}}%
{\XXint\scriptstyle\scriptscriptstyle{#1}}%
{\XXint\scriptscriptstyle\scriptscriptstyle{#1}}%
\!\int}
\def\XXint#1#2#3{{\setbox0=\hbox{$#1{#2#3}{\int}$ }
\vcenter{\hbox{$#2#3$ }}\kern-.6\wd0}}

\def\fint{\Xint-}

\makeatletter
\let\original@addcontentsline\addcontentsline
\newcommand{\dummy@addcontentsline}[3]{}
\newcommand{\DeactivateToc}{\let\addcontentsline\dummy@addcontentsline}
\newcommand{\ActivateToc}{\let\addcontentsline\original@addcontentsline}
\makeatother

\begin{document}

\title{Local null-controllability of a nonlocal semilinear heat equation}
\author{V\'ictor Hern\'andez-Santamar\'ia \and Kévin Le Balc'h }

\maketitle

\begin{abstract}
This paper deals with the problem of internal null-controllability of a heat equation posed on a bounded domain with Dirichlet boundary conditions and perturbed by a semilinear nonlocal term. We prove the small-time local null-controllability of the equation. The proof relies on two main arguments. First, we establish the small-time local null-controllability of a $2 \times 2$ reaction-diffusion system, where the second equation is governed by the parabolic operator $\tau \partial_t - \sigma \Delta$, $\tau, \sigma > 0$. More precisely, this controllability result is obtained uniformly with respect to the parameters $(\tau, \sigma) \in (0,1) \times (1, + \infty)$. Secondly, we observe that the semilinear nonlocal heat equation is actually the asymptotic derivation of the reaction-diffusion system in the limit $(\tau,\sigma) \rightarrow (0,+\infty)$. Finally, we illustrate these results by numerical simulations.
\end{abstract}
\small
\tableofcontents
\normalsize

\section{Introduction}

\subsection{Motivation}

Parabolic nonlocal equations have important applications in physics, biology, chemotaxis and ecology, see for instance the recent book \cite{KS18} where many models are introduced. The controllability of linear and nonlinear parabolic systems have been intensely studied in the past two decades, since the seminal papers of Lebeau, Robbiano \cite{LR95} and Fursikov Imanuvilov \cite{fursi} who prove independently the small-time null-controllability of the heat equation in any space dimension thanks to Carleman estimates. One can see the survey \cite{AKBGBdT11} and the recent thesis \cite{LB19} of the second author to get an overview of these results. Parabolic nonlocal models are a very challenging issue in the context of control theory. Indeed, even for linear equations, the by now classical Carleman estimates cannot handle in an easy way with the nonlocal terms. Let us mention a non exhaustive list of recent articles on the topic of controllability of nonlocal equations, see \cite{FCLZ16} for linear heat equation with an analytic nonlocal spatial term, \cite{LZ18} for linear systems, \cite{BHZ19} for linear and semilinear nonlocal heat equations, \cite{FCLNHNC19} for nonlocal nonlinear diffusion.

\subsection{Problem formulation and main results}
The goal of this part is to introduce into details the control problem that we will consider.\\
\indent Let $T >0$, $N \in  \{1,2,3\}$, $\Omega$ be a bounded, connected, open subset of $\R^{N}$ of class $C^2$, $\omega$ be a nonempty (small) open set contained in $\Omega$. We consider the semilinear heat equation with Dirichlet boundary conditions:
\begin{equation}
\label{eq:heatSL}
\begin{cases}
\D \partial_t y-  \Delta y = f\left(y(t,x), \fint_{\Omega} y (t,\xi) \d \xi\right) +  h 1_{\omega} &\mathrm{in}\ (0,T)\times\Omega,\\
y= 0&\mathrm{on}\ (0,T)\times\partial\Omega,\\
y(0,\cdot)=y_0& \mathrm{in}\ \Omega,
\end{cases}
\end{equation}
where $f \in C^{1}(\R^2;\R)$ and the nonlocal term is given by 
\[  \fint_{\Omega} y(t,\xi) \d \xi = \frac{1}{|\Omega|} \int_{\Omega} y(t,\xi) \d \xi.\]
In \eqref{eq:heatSL}, at time $t\in [0,T]$, $y(t,.): \Omega \rightarrow \R^2$ is the \textit{state} and $h(t,.)  : \Omega \rightarrow \R$ is the \textit{control input} supported in $\omega$.\\
\indent The question we ask in the following is a question of null-controllability at time $T$ for \eqref{eq:heatSL}, that is to say, given $T>0$ and an initial datum $y_0$, we wonder if there exists a control $h$ depending on time and space, locally supported in $\omega$, such that the corresponding solution $y$ of \eqref{eq:heatSL} vanishes at time $t=T$.\\
\indent The first main result of this article is a small-time local null-controllability result for \eqref{eq:heatSL}. For $\K = \R$ or $\R^2$, we denote by $W_0^{1, \infty}(\K)$ the set of functions $g \in W^{1, \infty}(\K)$ such that $\lim_{\infty} g = 0$.
\begin{theo}
\label{th:mainresult1}
Let $a,b \in \R$ and $g_1 \in W_0^{1,\infty}(\R^2)$, $g_2 \in W_0^{1,\infty}(\R)$. We assume that the nonlinearity $f$ writes
\begin{equation}
\label{eq:Hypf}
\forall (u,v) \in \R^2,\ f(u,v) = a u + b v + g_1(u,v) u^2 + g_2(u) uv.
\end{equation}
Then, \eqref{eq:heatSL} is locally null-controllable at any time $T>0$. More precisely, for every $T>0$, there exists $\delta >0$, such that for every $y_0 \in H_0^1(\Omega)$ satisfying $\norme{y_0}_{H_0^1(\Omega)} \leq \delta$, there exists $h \in L^{2}((0,T)\times \omega)$ such that the (unique) solution $y$ of \eqref{eq:heatSLLin} verifies
\begin{equation}
\label{eq:yT}
y(T,\cdot)=0.
\end{equation}
\end{theo}
\begin{rmk}
Let us discuss the (strong) assumption \eqref{eq:Hypf} that we make on the semilinearity. 
\begin{itemize}
\item First, $f(0,0) = 0$ ensures that $0$ is a stationary state of the free equation \eqref{eq:heatSL}, i.e. without control $h$. So, if we extend a control $h \in L^2((0,T)\times\omega)$ steering the solution $y$ of \eqref{eq:heatSL} to $0$ at time $T>0$ by $h \equiv 0$ in $(T,+\infty) \times \omega$, then the associated solution $y$ stays at $0$ for every $t \geq T$.
\item Secondly, by using \eqref{eq:Hypf}, we readily see that $f$ is globally Lipschitz so \eqref{eq:heatSL} is globally well-posed in $C([0,T];L^2(\Omega))$ for an initial datum $y_0 \in L^2(\Omega)$ and a control $h \in L^2((0,T)\times\omega)$, see \cite[Proposition 4.3]{HR00}.
\item By taking $g_1 = g_2 = 0$ in \eqref{eq:Hypf}, then $f$ is linear. So, from \Cref{th:mainresult1}, we can deduce a small-time (global) null-controllability result for \eqref{eq:heatSL} which is a generalization of \cite[Theorem 1.1 and Theorem 1.2]{MT18} to the multidimensional spatial case.
\item In \Cref{sec:formnonlinearity} below, we discuss the particular form of the nonlinearity, i.e. we explain why $g_2$ is not allowed to depend on the second variable $v$ and why we cannot add another nonlinear term $g_3(u,v) v^2$.
\end{itemize}
\end{rmk}
\begin{rmk}
We can also prove a small-time local null-controllability result with initial datum in $L^2(\Omega)$. Indeed, by setting $h \equiv 0$ in $(0,T/2) \times \omega$, then due to the regularizing effect of \eqref{eq:heatSL}, we obtain that $y(T/2,\cdot)$ is in $H_0^1(\Omega)$. So, we can apply \Cref{th:mainresult1} in the time interval $(T/2,T)$ to steer the solution $y$ of \eqref{eq:heatSL} to $0$ in time $T$.
\end{rmk}
Let us take an example inspired by \cite[Section 5.1.2]{Per15} from the theory of adaptive evolution on which \Cref{th:mainresult1} applies. By denoting $y(t,x)$ the density of individuals at time $t$, depending on a physiological parameter $x \in \Omega$, we assume that the total population compete and contribute and increase the death rate then satisfies \eqref{eq:heatSL} with $h=0$ and $f$ given by
\begin{equation}
\label{eq:fadaptive}
f\left(y, \fint y\right) =  \chi(y) y \left(B -  \fint y\right),
\end{equation}
where $\chi \in C_c^{\infty}(\R)$ such that $\chi \equiv 0$ in a neighbourhood of $0$, $B\in \R$. In \eqref{eq:fadaptive}, the term $B$ is the birth rate, which does not depend on the trait $x$ and $-  \fint y$ represents the death term, as in the Fisher/KPP equation. Then, in \eqref{eq:heatSL}, the Laplacian term takes into account mutations. We see that $f$ satisfies \eqref{eq:Hypf} with $a = b = 0$, $g_1(u,v) = B \frac{\chi(u)}{u}$, $g_2(u) = - \chi(u)$. So, \Cref{th:mainresult1} ensures the small-time local null-controllability of \eqref{eq:heatSL}, i.e. in terms of modelling, if the initial density of population is sufficiently small, then by acting on a specific location on the physiological parameter, one can ensure the extinction of the population in small time. On the other hand, an interesting open issue in order to be closed to modelling aspects would be to guarantee that the solution $y(t,x)$ stays non-negative, which is not ensured by \Cref{th:mainresult1}. It is possible that a minimal time of control appears in this case, as in \cite{LTZ17} for instance.

Our second main result, which is a by-product on the proof we follow for proving \Cref{th:mainresult1}, is an uniform small-time local null-controllability result for the $2 \times 2$ reaction-diffusion system
\begin{equation}
\label{eq:SystSL}
\left\{
\begin{array}{l l}
\partial_t u-  \Delta u = f\left(u,v\right) +  h 1_{\omega} &\mathrm{in}\ (0,T)\times\Omega,\\
\tau \partial_t v -  \sigma \Delta v = u-v  &\mathrm{in}\ (0,T)\times\Omega,\\
u = \frac{\partial v}{\partial n}= 0,\ &\mathrm{on}\ (0,T)\times\partial\Omega,\\
(u,v)(0,\cdot)=(u_0,v_0)& \mathrm{in}\ \Omega,
\end{array}
\right.
\end{equation}
where $n$ is the outer unit normal to $\partial\Omega$ and $(\tau, \sigma)$ are parameters in $(0,+\infty)$.
\begin{theo}
\label{th:mainresult2}
We assume that $f$ satisfies \eqref{eq:Hypf}. Then, \eqref{eq:SystSL} is uniformly with respect to $(\tau, \sigma) \rightarrow (0, +\infty)$ locally null-controllable at any time $T>0$. That is to say, for every time $T>0$, there exists $C,\ \delta >0$ such that for any $(\tau, \sigma) \in (0,1) \times (1,+\infty)$, for every initial data $(u_0,v_0) \in H_0^{1}(\Omega)\times H^1(\Omega)$ such that $\norme{(u_0,v_0)}_{H_0^{1}(\Omega)\times H^1(\Omega)} \leq \delta$, there exists $h_{\tau,\sigma} \in L^2((0,T)\times\omega)$ verifying
\begin{equation}
\label{eq:boundh}
\norme{h_{\tau,\sigma}}_{L^2((0,T)\times\omega)} \leq C,
\end{equation}
such that the unique solution $(u,v)_{\tau,\sigma}$ of \eqref{eq:SystSL} satisfies
\begin{equation}
\label{eq:uvT}
(u,v)_{\tau,\sigma}(T,\cdot)=0.
\end{equation}
\end{theo}
There are two main difficulties in \Cref{th:mainresult2}.
\begin{itemize}
\item The first one is to obtain the null-controllability of \eqref{eq:SystSL} by acting only on the first component of the system. This type of problem is by now classical from the seminal paper \cite{dT00}. Roughly speaking, in \eqref{eq:SystSL}, the control $h$ controls the component $u$ in the first equation and the component $u$ acts as an indirect control in the second equation through the coupling term $-u$ to control the second component $v$.
\item The second one is less classical because we require here to obtain uniform null-controllability results with respect to $(\tau, \sigma) \rightarrow (0, +\infty)$. This is handled by adapting proofs of \cite{CSGP14}, see also \cite{chaves_uniform} and \cite{K-S_felipe}.
\end{itemize}
\subsection{Strategy of the proof and bibliographical comments} 
In order to treat the local null-controllability of \eqref{eq:heatSL}, the natural strategy would be to linearize around $(0,0)$ to obtain
\begin{equation}
\label{eq:heatSLLin}
\begin{cases}
\D \partial_t y-  \Delta y = a y +  b \fint_{\Omega} y(t,\xi) \d \xi +  h 1_{\omega} &\mathrm{in}\ (0,T)\times\Omega,\\
y= 0&\mathrm{on}\ (0,T)\times\partial\Omega,\\
y(0,\cdot)=y_0& \mathrm{in}\ \Omega.
\end{cases}
\end{equation}
Except in the one-dimensional case, see \cite[Theorem 1.1 and Theorem 1.2]{MT18} where the authors establish the null-controllability of \eqref{eq:heatSLLin} by spectral techniques, we cannot handle with known results in the literature the null-controllability of \eqref{eq:heatSLLin}. Indeed, let us write the nonlocal term in \eqref{eq:heatSLLin} as the usual kernel form $$b \fint_{\Omega} y(t,\xi) \d \xi = \int_{\Omega} K(t,x, \xi) y(t,\xi) \d \xi,\qquad K (t,x,\xi) \equiv \frac{b}{|\Omega|}.$$ We cannot use \cite[Theorem 1.1]{BHZ19} because the kernel $K$ is constant, then does not decreases exponentially when $t \rightarrow T^{-}$. \cite[Theorem 3]{FCLZ16} also does not apply because the condition the authors require on the Fourier components of the kernel $K$ is not satisfied. Indeed, for instance by considering the Dirichlet Laplacian on $(0,\pi)$, we easily have for $m, j \geq 1$,
\begin{equation*}
k_{m,j} := \frac{b}{\pi} \int_{0}^{\pi} \int_{0}^{\pi} \sin(mx) \sin(jy) \d y \d x = \frac{b\left(1 - (-1)^m\right)\left(1 - (-1)^j\right)}{\pi mj}.
\end{equation*}
Then, by denoting $\lambda_k = k^2$ for $k \geq 1$, the sequence of eigenvalues of the Dirichlet Laplacian operator on $(0,\pi)$, we have for some constant $c>0$,
\begin{align*}
\norme{K}_{R}^2 := \sum_{m \geq 1} \left( \sum_{ j \geq 1} \lambda_j^{-1} k_{m,j}^2   \right)\lambda_{m}^{-1} e^{2 R \lambda_m^{1/2}}  \geq c \sum_{m \geq 1} \left( \sum_{ j \geq 1} \frac{1}{j^2} \frac{1}{m^2 j^2}  \right)\frac{1}{m^2} e^{2 R m} = + \infty.
\end{align*}

As a consequence, in order to prove \Cref{th:mainresult1}, we will use a non-straightforward approach. More precisely, we will first prove \Cref{th:mainresult2}. Then, roughly speaking, we will use the fact that the solution $(u_{\tau,\sigma}, v_{\tau,\sigma}, h_{\tau,\sigma})$ of \eqref{eq:SystSL} converges to $(y, \fint_{\Omega} y, h)$, the solution of \eqref{eq:heatSL} as $(\tau,\sigma)\rightarrow (0,+\infty)$, see \Cref{sec:passlimit}. This asymptotic result was first observed in \cite{HR00}, then extended in \cite{rodrigues}. Let us mention that the assumption \eqref{eq:Hypf} on $f$, which ensures that $f$ is globally Lipschitz, is crucial to pass to the limit.\\
\indent To prove \Cref{th:mainresult2}, our strategy is as follows.
\begin{itemize} 
\item First, we will linearize \eqref{eq:SystSL} around $(0,0)$. We then obtain the uniform null-controllability of the linearized system by proving an uniform observability inequality for the adjoint system thanks to Carleman estimates. This strategy is inspired by \cite{CSGP14}.
\item Secondly, we perform the source term method, introduced in \cite{LTT13}. That is to say, we obtain the (uniform) null-controllability of the linearized system to which we have added a source term, exponentially decreasing at $t \rightarrow T^{-}$.
\item Finally, we obtain the null-controllability of the nonlinear system by using a Banach fixed-point argument, similar to those employed in \cite{LTT13}.
\end{itemize}

\section{Uniform local null-controllability of the reaction-diffusion system}
\label{sec:uniformcontrol}

The goal of this section is to prove \Cref{th:mainresult2}.

\subsection{Uniform null-controllability of the linearized system} We linearize the system \eqref{eq:SystSL} around $((0,0), 0)$, then we obtain 
\begin{equation}
\label{eq:Syst_Linearized}
\begin{cases}
\D \partial_t u-  \Delta u = a  u +b  v +   h 1_{\omega} &\mathrm{in}\ (0,T)\times\Omega,\\
\tau \partial_t v -  \sigma \Delta v = u-v  &\mathrm{in}\ (0,T)\times\Omega,\\
\D u = \frac{\partial v}{\partial n}= 0,\ &\mathrm{on}\ (0,T)\times\partial\Omega,\\
(u,v)(0,\cdot)=(u_0,v_0)& \mathrm{in}\ \Omega,
\end{cases}
\end{equation}
where 
\begin{equation}
(a,b, c, d) = \left (\frac{\partial f}{\partial u}(0,0) , \frac{\partial f}{\partial v}(0,0), 1, -1\right).
\end{equation}
The goal of this part is to prove the uniform (global) null-controllability result of the linear reaction-diffusion system \eqref{eq:Syst_Linearized}.
We prove the uniform (global) null-controllability result of the linear reaction-diffusion system.
\begin{prop}
\label{prop:uniformNCLinear}
Let $(a,b,c,d) \in \R^2 \times \R^{*} \times (-\infty,0)$. For every $T\in(0,1)$, $(\tau,\sigma) \in (0,1) \times(1,+\infty)$, for every initial data $(u_0,v_0) \in L^2(\Omega)^2$, there exists a control $h \in L^2((0,T)\times\omega)$ satisfying
\begin{equation}
\norme{h}_{L^2((0,T)\times\omega)} \leq C_T \left( \norme{u_0}_{L^2(\Omega)} + \sqrt{\tau}\norme{u_0}_{L^2(\Omega)}\right) ,
\end{equation}
where 
\begin{equation}
\label{eq:CTcout}
C_T = C \exp \left( \frac{C}{T}\right)\ \text{with}\ C = C(\Omega,\omega, a, b, c, d) >0,
\end{equation}
such that the solution $(u,v)$ of \eqref{eq:Syst_Linearized} verifies $(u,v)(T,\cdot) = 0$.
\end{prop}

\begin{rmk}\label{rmk:unif_c_d}
We make some comments on the parameters $(a,b,c,d)$ of \Cref{prop:uniformNCLinear}.
\begin{itemize}
\item The condition $c \neq 0$ is a necessary condition of null-controllability of \eqref{eq:Syst_Linearized} because if $c = 0$ then the second equation of \eqref{eq:Syst_Linearized} is decoupled from the other so the component $v$ is not controllable. 
\item On the other hand, the condition $d < 0$ is necessary to ensure the uniform well-posedness of \eqref{eq:Syst_Linearized} in $C([0,T];L^2(\Omega)^2)$ with respect to the parameter $\tau >0$. Let us mention that the uniform null-controllability, with respect to the parameter $\sigma$, has already been proved by the first author and Enrique Zuazua in \cite{HSZ18} by only assuming that $c \neq 0$.
\end{itemize}
\end{rmk}
\subsection{Uniform observability estimate for the adjoint system}
In order to prove \Cref{prop:uniformNCLinear}, we will prove a uniform observability estimate for the adjoint system.
\begin{prop}\label{eq:unif_obs}
For every $T>0$, $(\tau,\sigma) \in (0,1) \times(1,+\infty)$, there exists a positive constant $C_T$ of the form \eqref{eq:CTcout} such that for every $(\phi_T, \psi_T) \in L^2(\Omega)^2$, we have
\begin{equation}
\label{eq:Obs}
\norme{\phi(0,\cdot)}_{L^2(\Omega)}^2 + \tau \norme{\psi(0,\cdot)}_{L^2(\Omega)}^2 \leq C_T \iint_{\omega\times(0,T)} |\phi(t,x)|^2 \dt \dx.
\end{equation}
where $(\phi, \psi)$ is the solution of the adjoint system
\begin{equation}
\label{eq:Syst_Linear_Adj}
\begin{cases}
- \partial_t  \phi -  \Delta \phi = a \phi + c \psi  &\mathrm{in}\ (0,T)\times\Omega,\\
- \tau \partial_t \psi  -  \sigma \Delta \psi = b \phi+ d \psi  &\mathrm{in}\ (0,T)\times\Omega,\\
\D \phi = \frac{\partial \psi}{\partial n}= 0,\ &\mathrm{on}\ (0,T)\times\partial\Omega,\\
(\phi,\psi)(T,\cdot)=(\phi_T,\psi_T)& \mathrm{in}\ \Omega.
\end{cases}
\end{equation}
\end{prop}

To prove \eqref{eq:Obs}, we use Carleman estimates in the spirit of \cite{CSGP14} where they obtain controllability results for fast diffusion coupled parabolic systems.
\subsection{Preliminaries on Carleman estimates}

We begin by recalling below a global Carleman estimate for heat equations with homogeneous Dirichlet or Neumann boundary conditions. For this, we introduce a special function whose existence is guaranteed by the following result (see \cite[Lemma 1.1]{fursi}).
\begin{lem}\label{eta_fursi}
Let ${\mathcal B}\subset\subset\Omega$ be a nonempty open subset. Then, there exists $\eta^0\in C^2(\overline{\Omega})$ such that $\eta^0>0$ in $\Omega$, $\eta^0=0$ on $\partial \Omega$, and $|\nabla \eta^0|>0$ in $\overline{\Omega\setminus\mathcal B}$.
\end{lem}

Then, for a parameter $\lambda>0$, we introduce the weight functions
\begin{gather}\notag %
\alpha(x,t)=\frac{e^{2\lambda\|\eta^0\|_\infty}-e^{\lambda\eta^0(x)}}{t(T-t)}, \quad \xi(x,t)=\frac{e^{\lambda\eta^0(x)}}{t(T-t)}, \\ \label{weights}
 \alpha^\star(t)=\max_{x\in \overline \Omega}\alpha(x,t), \quad \widehat\alpha(t)=\min_{x\in \overline \Omega}\alpha(x,t), \quad \widehat\xi(t)=\max_{x\in\overline \Omega}\xi(x,t), \quad \xi^\star(t)=\min_{x\in\overline \Omega} \xi(x,t). 
\end{gather}
To abridge the presentation of the estimates below we will use the notation
\begin{equation}\label{eq:notation_car_est}
\begin{split}
I(q;\epsilon,\beta):=& \ s^{\beta-1}\iint_Qe^{-2s\alpha}\xi^{\beta-1}\left(\epsilon^2|\partial_t q|^2+\sum_{i,j=1}^{N}\left|\frac{\partial^2q}{\partial x_i\partial x_j}\right|^2\right)\dx\dt\\
&+s^{\beta+1}\iint_{Q}e^{-2s\alpha}\xi^{\beta+1}|\nabla q|^2\dx\dt+s^{\beta+3}\iint_{Q}e^{-2s\alpha}\xi^{\beta+3}|q|^2\dx\dt,
\end{split}
\end{equation}
for some positive parameters $s$, $\beta$ and $\epsilon$. Hereinafter, we use the notations $Q:= (0,T)\times\Omega$ and $\Sigma := (0,T)\times\partial\Omega$. 

The Carleman inequality reads as follows (see \cite[Lemma 2.1]{CSGP14}).
\begin{lem}\label{car_refined}
There exist two positive constants $C=C(\Omega,\mathcal B)$ and $\lambda_0=\lambda_0(\Omega,\mathcal B)$ such that, for every $\lambda\geq \lambda_0$, there exists $s_0=s_0(\Omega,\mathcal B,\lambda)$ such that, for any $s\geq s_0(T+T^2)$, $q_T\in L^2(\Omega)$ and $g\in L^2(Q)$, the weak solution to
\begin{equation*}
\begin{cases}
-\epsilon \partial_t q-\Delta q=g(x,t) & \text{in } Q, \\
\D \frac{\partial q}{\partial n}= \ (\textnormal{or } q=)\ 0 &\text{on } \Sigma ,\\
q(T,x)=q_T(x) & \text{in } \Omega,
\end{cases}
\end{equation*}
satisfies
\begin{equation}\label{car_sigma}
I(q;\epsilon,\beta)\leq C\left(s^{\beta}\iint_{Q}e^{-2s\alpha}\xi^{\beta}|g|^2\dx\dt+s^{\beta+3}\iint_{\mathcal B\times(0,T)}e^{-2s\alpha}\xi^{\beta+3}|q|^2\dx\dt \right),
\end{equation}
for all $\beta\in\mathbb R$ and any $0<\epsilon\leq 1$.
\end{lem}

The proof of Lemma \ref{car_refined} can be deduced from the Carleman inequality for the heat equation with homogeneous Neumann boundary (or Dirichlet) conditions given in \cite[Theorem 1]{fc_b_g_p} or \cite{fursi} and arguing as in \cite[Appendix]{chaves_uniform}. We notice that the weights for both cases, Neumann and Dirichlet boundary conditions, are the same and are fully compatible.
\subsection{Uniform global Carleman estimate for the adjoint system}

Now, we are in position to state the main result of this section, as to say a Carleman estimate for the adjoint system \eqref{eq:Syst_Linear_Adj} given by the following theorem.

\begin{theo}\label{thm:main_carleman}
Assume that $c\neq 0$ and $d<0$. Given any $(\tau,\sigma)\in (0,1) \times (1, +\infty)$, there exist positive constants $C=C(\Omega,\omega)$ and $\lambda_1=\lambda_1(\Omega,\omega)$ such that, for every $\lambda\geq \lambda_1$, there exists $s_2=s_2(\Omega,\omega,\lambda,a,b,c,d)$ such that, for any $s\geq s_2(T+T^2)$, $(\phi_T,\psi_T)\in [L^2(\Omega)]^2$, the solution to \eqref{eq:Syst_Linear_Adj} satisfies
\begin{equation}\label{car_k_libre}
\begin{split}
s^4\iint_Qe^{-2s\alpha}\xi^4|\phi|^2&\dx\dt+s^4\iint_Qe^{-2s\alpha}\xi^4|\psi|^2\dx\dt  \\
&\leq Cs^{14}\iint_{\omega\times(0,T)}\big(e^{-2s\widehat\alpha}+e^{-4s\widehat \alpha+2s\alpha^\star}\big)(\widehat\xi)^{14}|\phi|^2\dx\dt.
\end{split}
\end{equation}
\end{theo}
For proving \Cref{thm:main_carleman}, we will extend the adjoint system \eqref{eq:Syst_Linear_Adj} to a system of four equations. We define
\begin{equation*}
-\frac{\tau}{\sigma}\partial_t\phi-\Delta \phi -\frac{d}{\sigma}\phi=:w.
\end{equation*}
If $\phi_T$, $\psi_T$ belong to $C_0^\infty(\Omega)$ and $(\phi,\psi)$ is the solution to \eqref{eq:Syst_Linear_Adj} associated to this initial data, then a straightforward computation yields
\begin{equation*}
-\partial_t w -\Delta w-aw= \frac{cb}{\sigma}\phi\ \text{in}\ (0,T)\times\Omega\ \text{and}\ w=0\ \text{on}\ (0,T)\times\partial\Omega.
\end{equation*}
Thus, we can extend our initial adjoint system to a system of four equations given by
\begin{equation}\label{eq:extended_adj}
\begin{cases}
\D -\partial_t w -\Delta w - aw = \frac{cb}{\sigma}\phi & \text{in } Q, \\
\D -\frac{\tau}{\sigma}\partial_t\phi - \Delta\phi-\frac{d}{\sigma}\phi=w & \text{in } Q, \\
- \partial_t  \phi -  \Delta \phi = a \phi + c \psi  &\text{in } Q,\\
- \tau \partial_t \psi  -  \sigma \Delta \psi = b \phi+ d \psi  &\text{in } Q,\\
\D w=\phi= \frac{\partial \psi}{\partial n}= 0\ &\text{on } \Sigma,\\
(w,\phi,\psi)(T,\cdot)=(-\Delta\phi_T-\tfrac{d}{\sigma}\phi_T,\phi_T,\psi_T)& \mathrm{in}\ \Omega.
\end{cases}
\end{equation}

The outline of the proof follows some of the ideas in \cite{CSGP14} but since the coupling in our system is through constant coefficients, the computations are greatly simplified. For clarity, we have divided the proof into five parts, which can be summarized as follows:
\begin{itemize}
\item First part: we apply the Carleman inequality \eqref{car_sigma} to the first, third and fourth equations of system \eqref{eq:extended_adj}. This gives global estimates for $w$, $\phi$ and $\psi$ in terms of local terms of $w$, $\phi$ and $\psi$.
\item Second part: using the second equation in \eqref{eq:extended_adj}, we estimate locally $w$ coming from the previous step.
\item Third part: in this step, we will estimate the local term of $\psi$ in terms of local integrals of the variables $\phi$ and $\partial_t\phi$ and some lower order terms. 
\item Fourth part: following \cite{CSGP14}, we will estimate the local term of $\partial_t\phi$ by means of sharp weighted estimates. Here, we will obtain a local integral of $\phi$ and several lower order terms.
\item Fifth part: combining the different estimates obtained in the previous steps, we will absorb the lower order terms yielding the desired Carleman inequality \eqref{car_k_libre}. 
\end{itemize}

\begin{proof}[Proof of \Cref{thm:main_carleman}]
Let us consider $\omega_i\subset \Omega$, $i=0,1$ such that $\omega_0\subset\subset\omega_1\subset\subset\omega$. In what follows, $C$ will denote a generic positive constant only depending on $\Omega$ and $\omega$ that may change from line to line. 
\subsubsection*{Step 1. First estimates} 
We apply the estimate \eqref{car_sigma} to the first equation of \eqref{eq:extended_adj} with $\mathcal B=\omega_0$, $g=aw+\tfrac{cb}{\sigma}\phi$ and $\beta=2$ to get
\begin{equation}\label{eq:car_w_init}
I(w;1,2)\leq C\left(s^{2}\iint_{Q}e^{-2s\alpha}\xi^{2}|aw+\tfrac{cb}{\sigma}\phi|^2\dx\dt+s^{5}\iint_{\omega_0\times(0,T)}e^{-2s\alpha}\xi^{5}|w|^2\dx\dt \right).
\end{equation}

We do the same procedure for the third equation with $\mathcal B=\omega_0$, $\beta=1$ and its respective right-hand side, thus
\begin{equation}\label{eq:car_phi_init}
I(\phi;1,1)\leq C\left(s\iint_{Q}e^{-2s\alpha}\xi|a\phi+c\psi|^2\dx\dt+s^{4}\iint_{\omega_0\times(0,T)}e^{-2s\alpha}\xi^{4}|\phi|^2\dx\dt \right).
\end{equation}
Now, we divide over $\sigma$ in the fourth equation of \eqref{eq:extended_adj} and apply inequality \eqref{car_sigma} with $g=\sigma^{-1}(b\phi+d\psi)$, $\mathcal B=\omega_0$ and $\beta=1$, that is,
\begin{equation}\label{eq:car_psi_init}
I(\psi;\tfrac{\tau}{\sigma},1)\leq C\left(s\iint_{Q}e^{-2s\alpha}\xi|\sigma^{-1}(b\phi+d\psi)|^2\dx\dt+s^{4}\iint_{\omega_0\times(0,T)}e^{-2s\alpha}\xi^{4}|\psi|^2\dx\dt \right).
\end{equation}

Adding inequalities \eqref{eq:car_w_init}--\eqref{eq:car_psi_init} and using that $\sigma^{-1}\leq 1$, we can absorb the lower order terms in the right hand side of the inequalities with the parameter $s$. More precisely, we obtain
\begin{equation}\label{eq:car_sum_init}
\begin{split}
&I(w;1,2)+I(\phi;1,1)+I(\psi;\tfrac{\tau}{\sigma},1) \leq C\left(s^{5}\iint_{\omega_0\times(0,T)}e^{-2s\alpha}\xi^{5}|w|^2\dx\dt  \right.\\
&\qquad \left. + s^{4}\iint_{\omega_0\times(0,T)}e^{-2s\alpha}\xi^{4}|\phi|^2\dx\dt + s^{4}\iint_{\omega_0\times(0,T)}e^{-2s\alpha}\xi^{4}|\psi|^2\dx\dt \right),
\end{split}
\end{equation}
for every 
\begin{equation}\label{eq:s_form}
s\geq s_1(T+T^2)
\end{equation}
where $s_1$ is a positive constant only depending on $\Omega$, $\omega_0$ and the coefficients $a,b,c,d$.

\subsubsection*{Step 2. Local estimate for $w$}
Here, we estimate the local integral of $w$ in the right-hand side of \eqref{eq:car_sum_init}. We will do this by using the second equation of system \eqref{eq:extended_adj}. To this end, consider a function $\eta\in C_0^\infty(\omega_1)$ satisfying 
\begin{equation}\label{eq:cut}
0\leq \eta\leq 1 \quad\text{and}\quad \eta(x)\equiv 1 \; \text{for all } x\in\omega_0.
\end{equation}
Then, we obtain 
\begin{align} \notag
 s^5\iint_{\omega_0\times(0,T)}e^{-2s\alpha}\xi^5|w|^2\dx\dt&\leq s^5\iint_{\omega_1\times(0,T)}e^{-2s\alpha}\xi^5  |w|^2\eta\dx\dt\\ \notag
 &=s^5\iint_{\omega_1\times(0,T)}e^{-2s\alpha}\xi^5 w\left(-\frac{\tau}{\sigma}\partial_t\phi-\Delta\phi-\frac{d}{\sigma}\phi\right)\eta\dx\dt \\ \label{eq:Mis}
 &:= M_1+M_2+M_3. 
\end{align}

Let us estimate each term $M_i$, $1\leq i\leq 3$. Integrating by parts in time, we see that
\begin{equation*}
\begin{split}
M_1=&-2\frac{\tau}{\sigma} s^6 \iint_{\omega_1\times(0,T)}e^{-2s\alpha}\alpha_t\xi^5w\phi \eta \dx\dt+5\frac{\tau}{\sigma}s^5\iint_{\omega_1\times(0,T)}e^{-2s\alpha}\xi^4\xi_tw\phi\dx\dt \\
&+\frac{\tau}{\sigma}s^5\iint_{\omega_1 \times(0,T)}e^{-2s\alpha}\xi^5\omega_t\phi\eta.
\end{split}
\end{equation*}
Taking into account that $|\alpha_t|\leq C T\xi^2$ and $\frac{\tau}{\sigma}\leq 1$, we can apply Cauchy-Schwarz and Young inequalities to deduce
\begin{equation*}
\begin{split}
|M_1|\leq &\, \delta\left(s^5\iint_{\omega_1\times(0,T)}e^{-2s\alpha}\xi^5|w|^2\dx\dt+s\iint_{\omega_1\times(0,T)}e^{-2s\alpha}\xi|\partial_t w|^2\dx\dt\right) \\
&+ C_\delta\left(T^2s^7\iint_{\omega_1\times(0,T)}e^{-2s\alpha}\xi^9|\phi|^2\dx\dt+T^2s^5\iint_{\omega_1\times(0,T)}e^{-2s\alpha}\xi^7|\phi|^2 \right. \\
&\qquad\quad +\left. s^9\iint_{\omega_1\times(0,T)}e^{-2s\alpha}\xi^9|\phi|^2\dx\dt \right)
\end{split}
\end{equation*}
for any $\delta>0$ small enough. Then, using \eqref{eq:s_form} we can simplify the above expression as
\begin{equation}\label{eq:M1_est}
\begin{split}
|M_1|\leq & \delta\left(s^5\iint_{\omega_1\times(0,T)}e^{-2s\alpha}\xi^5|w|^2\dx\dt+s\iint_{\omega_1\times(0,T)}e^{-2s\alpha}\xi|\partial_t w|^2\dx\dt\right) \\
&+ C_\delta \left( s^9\iint_{\omega_1\times(0,T)}e^{-2s\alpha}\xi^9|\phi|^2\dx\dt \right).
\end{split}
\end{equation}

On the other hand, integrating by parts in the space variable, we readily get
\begin{equation*}
M_2=s^5\iint_{\omega_1\times(0,T)}e^{-2s\alpha}\xi^5 \nabla w\cdot\nabla\phi \eta \dx\dt+s^5\iint_{\omega_1\times(0,T)}\nabla(e^{-2s}\alpha\xi^5\eta)\cdot\nabla\phi w\dx\dt
\end{equation*}
and using that $|\nabla(e^{-2s\alpha}\xi^5)|\leq Cse^{-2s\alpha}\xi^6$ together with the properties of the function $\eta$, it is not difficult to see that
\begin{equation}\label{eq:M2_est}
\begin{split}
|M_2|\leq &\delta \left(s^5\iint_{\omega_1\times(0,T)}e^{-2s\alpha}\xi^5|w|^2\dx\dt+s^3\iint_{\omega_1\times(0,T)}e^{-2s\alpha}\xi^3|\nabla w|^2\dx\dt\right) \\
& + C_\delta\left(s^7\iint_{\omega_1\times(0,T)}e^{-2s\alpha}\xi^7|\nabla\phi|^2\dx\dt\right)
\end{split}
\end{equation}
for all $\delta>0$. 

The term $M_3$ can be easily bounded as
\begin{equation}\label{eq:M3_est}
|M_3|\leq \delta s^{5}\iint_{\omega_1\times(0,T)}e^{-2s\alpha}\xi^5|w|^2\dx\dt+C_{\delta}s^5\iint_{\omega_1\times(0,T)}e^{-2s\alpha}\xi^5|\phi|^2\dx\dt
\end{equation}
where we have used that $\sigma^{-1}\leq 1$. Using estimates \eqref{eq:M1_est}--\eqref{eq:M3_est} in \eqref{eq:Mis} and taking into account definition \eqref{eq:notation_car_est}, we get
\begin{equation}\label{eq:estw_local}
\begin{split}
&s^5\iint_{\omega_0\times(0,T)}e^{-2s\alpha}\xi^5|w|^2\dx\dt \\
&\leq C_{\delta}\left(s^7\iint_{\omega_1\times(0,T)}e^{-2s\alpha}\xi^7|\nabla\phi|^2\dx\dt+s^9\iint_{\omega_1\times(0,T)}e^{-2s\alpha}\xi^9|\phi|^2\dx\dt\right) + \delta I(w;1,2).
\end{split}
\end{equation}

To estimate the local integral of $\nabla\phi$ in the above inequality, we consider another cut-off function $\tilde \eta\in C_0^\infty(\omega)$ satisfying $0\leq \tilde\eta\leq 1$, $\tilde\eta\equiv 1$ on $\omega_1$. Integration by parts yields
\begin{equation*}
\begin{split}
s^7\iint_{\omega\times(0,T)}e^{-2s\alpha}\xi^7\tilde{\eta}|\nabla\phi|^2\dx\dt=&-s^7\iint_{\omega\times(0,T)} e^{-2s\alpha}\xi^7\tilde{\eta} \Delta\phi\phi\dx\dt \\
&+\frac{1}{2}\iint_{\omega\times(0,T)}\Delta(e^{-2s\alpha}\xi^7\tilde\eta)|\phi|^2\dx\dt.
\end{split}
\end{equation*}
Since $|\Delta(\tilde \eta e^{-2s\alpha}\xi^7)|\leq Cs^2\xi^9e^{-2s\alpha}$ in $\omega\times(0,T)$, we can apply using Cauchy-Schwarz and Young inequalities to obtain
\begin{equation}\label{eq:est_nabla}
\begin{split}
s^7\iint_{\omega_1\times(0,T)}e^{-2s\alpha}\xi^7|\nabla\phi|^2\dx\dt\leq&\ \delta \iint_{\omega\times(0,T)}e^{-2s\alpha}|\Delta\phi|^2\dx\dt \\
&+C_\delta s^{14}\iint_{\omega\times(0,T)}e^{-2s\alpha}\xi^{14}|\phi|^2\dx\dt,
\end{split}
\end{equation}
for all $\delta >0$. Therefore, combining \eqref{eq:estw_local} and \eqref{eq:est_nabla}, we have
\begin{equation}\label{eq:estw_final}
\begin{split}
s^5\iint_{\omega_0\times(0,T)}&e^{-2s\alpha}\xi^5|w|^2\dx\dt \\
&\leq \delta\left[I(w;1,2)+I(\phi;1,1)\right]+C_\delta s^{14}\iint_{\omega\times(0,T)}e^{-2s\alpha}\xi^{14}|\phi|^2\dx\dt.
\end{split}
\end{equation}

Using \eqref{eq:estw_final} in the right-hand side of \eqref{eq:car_sum_init} and taking $\delta$ sufficiently small, we obtain
\begin{equation}\label{eq:car_sum_step2}
\begin{split}
I(w;1,2)+&I(\phi;1,1)+I(\psi;\tfrac{\tau}{\sigma},1) \\
&\leq C\left(s^{14}\iint_{\omega_0\times(0,T)}e^{-2s\alpha}\xi^{14}|\phi|^2\dx\dt + s^{4}\iint_{\omega_0\times(0,T)}e^{-2s\alpha}\xi^{4}|\psi|^2\dx\dt \right).
\end{split}
\end{equation}
for all $s\geq s_1(T+T^2)$. 

\subsubsection*{Step 3. Local estimate for $\psi$}

In this step, we estimate the local integral of $\psi$ in inequality \eqref{eq:car_sum_step2} by local terms in the variable $\phi$ and $\phi_t$ and some other lower order terms. The idea of leaving the local term corresponding to $\phi_t$ has been already used in \cite{CSGP14} and \cite{HSZ18}. 

By considering the cut-off function $\eta$ given in \eqref{eq:cut} and the definition of the weight functions \eqref{weights}, we observe using the third equation of system \eqref{eq:extended_adj} that
\begin{align} \notag
 s^4\iint_{\omega_1\times(0,T)}e^{-2s\widehat{\alpha}}(\widehat\xi)^4  |\psi|^2\eta\dx\dt&=\frac{1}{c}s^4\iint_{\omega_1\times(0,T)}e^{-2s\widehat\alpha}(\widehat \xi)^4 \psi\left(-\partial_t\phi-\Delta\phi-a\phi\right)\eta\dx\dt \\ \label{eq:M5is}
 &=: M_4+M_5+M_6. 
\end{align}
Observe that at this point is crucial to have $c\neq 0$. Also notice that the weights in the above expression are $x$-independent.

As in the previous step, we will estimate each term in the above equation. For the first one, we readily have
\begin{equation}\label{eq:M4}
|M_4|\leq \frac{1}{2}s^4\iint_{\omega_1\times(0,T)}e^{-2s\widehat\alpha}(\widehat{\xi})^4|\psi|^2\eta\dx\dt+\frac{1}{2c^2}s^4\iint_{\omega_1\times(0,T)}e^{-2s\widehat\alpha}(\widehat{\xi})^4|\partial_t\phi|^2\eta\dx\dt.
\end{equation}

Integrating by parts twice in space yields
\begin{align*}
M_5=&-\frac{1}{c}s^4\iint_{\omega_1\times(0,T)}e^{-2s\widehat{\alpha}}(\widehat\xi)^4\Delta\psi\phi\eta\dx\dt-\frac{2}{c}s^4\iint_{\omega_1\times(0,T)}e^{-2s\widehat{\alpha}}(\widehat\xi)^4\nabla\eta\cdot\nabla \psi \phi\dx\dt \\
&-\frac{1}{c}s^4\iint_{\omega_1\times(0,T)}e^{-2s\widehat{\alpha}}(\widehat\xi)^4 \psi\phi\Delta\eta\dx\dt.
\end{align*}
Using the above equality, we can show using Cauchy-Schwarz and Young inequalities and the properties of the cut-off function that
\begin{equation}\label{eq:M56}
\begin{split}
&|M_5|+|M_6|\leq \delta  \left(\iint_{\omega_1\times(0,T)}e^{-2s\alpha}|\Delta\psi|^2\dx\dt+s^2\iint_{\omega\times(0,T)}e^{-2s\alpha}\xi^2|\nabla\psi|^2\dx\dt\right. \\
&\qquad \qquad\quad \left.+s^4\iint_{\omega_1\times(0,T)}e^{-2s\alpha}\xi^4|\psi|^2\dx\dt\right)+ C_\delta s^8\iint_{\omega_1\times(0,T)}e^{-4s\widehat\alpha+2s\alpha}(\widehat\xi)^8|\phi|^2\dx\dt.
\end{split}
\end{equation}
Here, we have used that $\widehat{\xi}\leq C\xi$ for some $C>0$ only depending on $\Omega$ and $\omega$. Since 
\begin{equation*}
 s^4\iint_{\omega_0\times(0,T)}e^{-2s\alpha}\xi^4|\psi|^2\dx\dt\leq s^4\iint_{\omega_1\times(0,T)}e^{-2s\widehat{\alpha}}(\widehat\xi)^4 |\psi|^2\eta\dx\dt
\end{equation*}
we can put together \eqref{eq:M5is} and estimates \eqref{eq:M4}--\eqref{eq:M56} to deduce
\begin{equation}\label{eq:est_psi_local}
\begin{split}
& s^4\iint_{\omega_0\times(0,T)}e^{-2s\alpha}\xi^4|\psi|^2\dx\dt \leq \delta I(\psi;1,1) \\
 &\qquad  +C_\delta \left(s^4\iint_{\omega_1\times(0,T)}e^{-2s\widehat\alpha}(\widehat{\xi})^4|\partial_t\phi|^2\eta\dx\dt+ s^8\iint_{\omega_1\times(0,T)}e^{-4s\widehat\alpha+2s\alpha}(\widehat\xi)^8|\phi|^2\dx\dt\right).
 \end{split}
 \end{equation}

Using estimate \eqref{eq:est_psi_local} in inequality \eqref{eq:car_sum_step2} with $\delta$ small enough and since $\omega_1\subset\omega$, we finally obtain
\begin{equation}\label{eq:car_sum_step3}
\begin{split}
&I(w;1,2)+I(\phi;1,1)+I(\psi;\tfrac{\tau}{\sigma},1) \\
&\leq C\left(s^4\iint_{\omega\times(0,T)}e^{-2s\widehat\alpha}(\widehat{\xi})^4|\partial_t\phi|^2\eta\dx\dt+ s^8\iint_{\omega\times(0,T)}e^{-4s\widehat\alpha+2s\alpha}(\widehat\xi)^8|\phi|^2\dx\dt \right),
\end{split}
\end{equation}
for all $s\geq s_1(T+T^2)$. 

\subsubsection*{Step 4. Local estimate for $\partial_t\phi$}
In this step, we deal with the second term appearing on the right-hand side of \eqref{eq:car_sum_step3}. Integrating by parts in the time variable yields
\begin{equation}\label{eq:est_init_phi_t}
\begin{split}
s^4\iint_{\omega\times(0,T)}e^{-2s\widehat\alpha}(\widehat\xi)^4|\partial_t\phi|^2=&-s^4\iint_{\omega\times(0,T)}e^{-2s\widehat{\alpha}}(\widehat{\xi})^4\phi_{tt} \phi\dx\dt \\
&+\frac{s^4}{2}\iint_{\omega\times(0,T)}\big(e^{-2s\widehat \alpha}(\widehat{\xi})^4\big)_{tt}|\phi|^2\dx\dt,
\end{split}
\end{equation}
and since 
\begin{equation}\label{eq:est_phi_tt}
\begin{split}
s^4\iint_{\omega_\times(0,T)}e^{-2s\widehat \alpha}(\widehat \xi)^4\phi_{tt}\phi\dx\dt\leq&\, \frac{s^{-6}}{2}\iint_{\omega_{\times(0,T)}}e^{-2s\alpha^\star}(\xi^\star)^{-6}|\partial_{tt}\phi|^2\dx\dt \\
&+\frac{s^{14}}{2}\iint_{\omega\times(0,T)}e^{-4s\widehat{\alpha}+2s\alpha^\star}(\xi^\star)^6(\widehat\xi)^8|\varphi^2|\dx\dt,
\end{split}
\end{equation}
it is enough to estimate the integral of $\phi_{tt}$ in the right-hand side of \eqref{eq:est_phi_tt}. Observe that we have introduced the smaller weight function $e^{-2s\alpha^\star}$ which only depends on time.

We define
\begin{equation}\label{eq:change_u}
u:=e^{-s\alpha^\star}(\xi^\star)^{-6/2}\partial_{tt}\phi
\end{equation}
Then, using the second equation of system \eqref{eq:extended_adj}, it not difficult to see that $u$ verifies
\begin{equation}\label{eq:sys_u}
\begin{cases}
\D -\frac{\tau}{\sigma}\partial_tu-\Delta u=\frac{d}{\sigma}u-\frac{\tau}{\sigma}\big(e^{-s\alpha^\star}(\xi^\star)^{-6/2}\big)_t\partial_{tt}\phi+e^{-s\alpha^\star}(\xi^\star)^{-6/2}\partial_{tt}w &\text{in }Q, \\
\D u=0 &\text{on } \Sigma, \\
u(0,\cdot)=u(T,\cdot)=0 &\text{in } \Omega.
\end{cases}
\end{equation}
Multiplying by $u$ and integrating in $L^2(\Omega)$, we get
\begin{equation}\label{eq:energy_u}
\begin{split}
-\frac{\tau}{\sigma}\frac{\d}{\dt}\left(\int_{\Omega}|u|^2\dx\right)+\int_{\Omega}|\nabla u|^2\dx=&\ \frac{d}{\sigma}\int_{\Omega}|u|^2\dx-\frac{\tau}{\sigma}\int_{\Omega} \big(e^{-s\alpha^\star}(\xi^\star)^{-6/2}\big)_t(\partial_{tt}\phi) u\dx \\
&+\int_{\Omega}e^{-s\alpha^\star}(\xi^\star)^{-6/2}(\partial_{tt}w) u\dx .
\end{split}
\end{equation}
Integrating in $[0,T]$ and since the first term in the right-hand side of \eqref{eq:energy_u} is negative, we obtain using Cauchy-Schwarz and Young inequalities 
\begin{align*}
\iint_{Q}|\nabla u|^2\dx\dt \leq & \ 2\delta \iint_{Q}|u|^2\dx\dt+C_{\delta}\frac{\tau^2}{\sigma^2}\iint_{Q}\big|(e^{-s\alpha^\star(\xi^\star)^{-6/2}})_t\partial_{tt}\phi\big|^2\dx\dt \\
&+C_{\delta}\iint_{Q}e^{-2s\alpha^\star}(\xi^{\star})^{-6}|\partial_{tt}w|^2\dx\dt
\end{align*}
for $\delta>0$. As in \cite{CSGP14}, we can absorb the term corresponding to $u$ by means of Poincar\'e inequality. Indeed, by taking $\delta\ll\mu_1$ where $\mu_1$ is the first eigenvalue of the Dirichlet Laplacian operator, we recover by recalling the change of variables \eqref{eq:change_u} that
\begin{equation}\label{eq:est_u_fin}
\begin{split}
s^{-6}\iint_{Q}e^{-2s\alpha^\star}(\xi^\star)^{-6}|\partial_{tt}\phi|^2\dx\dt \leq &\ C \frac{\tau^2}{\sigma^2}s^{-6}\iint_{Q}\big|(e^{-s\alpha^\star(\xi^\star)^{-6/2}})_t\partial_{tt}\phi\big|^2\dx\dt\\
&+C s^{-6}\iint_{Q}e^{-2s\alpha^\star}(\xi^{\star})^{-6}|\partial_{tt}w|^2\dx\dt
\end{split}
\end{equation}
for some constant $C>0$ uniform with respect to $\tau$ and $\sigma$. 
\begin{rmk}\label{rmk:important}
If both components have Neumann boundary conditions, the previous argument works only if $\fint_{\Omega} u = 0$, i.e. $\fint_{\Omega} \phi = 0$ because in this case the Poincaré inequality holds. If not, it seems that we have to use the first term in the right-hand side of \eqref{eq:energy_u}. Indeed, after integration in time, we recover
\begin{equation}\label{eq:energy_u_neum}
\begin{split}
\iint_{Q}|\nabla u|^2\dx\dt-\underbrace{\frac{d}{\sigma}\iint_{Q}|u|^2\dx\dt}_{\geq 0}=&-\frac{\tau}{\sigma}\iint_{Q} \big(e^{-s\alpha^\star}(\xi^\star)^{-6/2}\big)_t(\partial_{tt}\phi) u\dx\dt \\
&+\iint_{Q}e^{-s\alpha^\star}(\xi^\star)^{-6/2}(\partial_{tt}w) u\dx\dt .
\end{split}
\end{equation}
Observe that the quadratic term of $u$ is positive since $d<0$, but it is divided over $\sigma$. Then by Cauchy-Schwarz and Young inequalities,
\begin{equation}\label{eq:est_u_neum}
\begin{split}
\int_{Q}|\nabla u|^2\dx\dt-\underbrace{\frac{d}{\sigma}\int_{Q}|u|^2\dx\dt}_{\geq 0}&\leq \frac{\delta}{\sigma} \int_{Q}|u|^2\dx\dt + \frac{\tau^2}{\sigma}\int_{Q}\big|(e^{-s\alpha^\star(\xi^\star)^{-6/2}})_t\partial_{tt}\phi\big|^2\dx\dt \\
&\qquad+\frac{\delta}{\sigma}\int_{Q}|u|^2\dx\dt+\sigma\int_{Q}e^{-2s\alpha^\star}(\xi^\star)^{-6}|\partial_{tt}w|^2 \dx\dt .
\end{split}
\end{equation}
We highlight the fact that we have to introduce the parameter $\sigma$ while performing Young's inequality in the second term of \eqref{eq:energy_u_neum} to make both sides comparable. By taking $\delta$ small enough, we can absorb the right hand side terms of \eqref{eq:est_u_neum} containing a factor $\delta$, then obtain
\begin{equation}\label{eq:est_fin_neum}
\begin{split}
\iint_{Q}|u|^2\dx\dt \leq & C \tau^2\iint_{Q}\big|(e^{-s\alpha^\star}(\xi^\star)^{-6/2})_t\partial_{tt}\phi\big|^2\dx\dt \\
&+C\sigma^2 \iint_{Q}e^{-2s\alpha^\star}(\xi^\star)^{-6}|\partial_{tt}w|^2 \dx\dt 
\end{split}
\end{equation}
for some constant $C>0$ independent of $\sigma$ and $\tau$. If we compare \eqref{eq:est_u_fin} and \eqref{eq:est_fin_neum}, we will see using the definition of $u$ that they are almost the same, but the second term containing $\partial_{tt}w$ now has a factor $\sigma^2 >> 1$. We do not manage to overcome this new difficulty. This ends \Cref{rmk:important}.
\end{rmk}

Now, let us estimate in the right-hand side of \eqref{eq:est_u_fin}. To this end, we will use the second and third equations of system \eqref{eq:extended_adj}. Differentiating with respect to time in both equations we get
\begin{align}\label{eq:deriv_first}
-\frac{\tau}{\sigma}&\partial_{tt}\phi-\Delta\phi_t-\frac{d}{\sigma}\phi_t=w_t,\\ \label{eq:deriv_second}
-&\partial_{tt}\phi-\Delta\phi_t=a\phi_t+c\psi_t
\end{align}
Multiplying \eqref{eq:deriv_first} and \eqref{eq:deriv_second} by $\tau/\sigma$ then subtracting, we can obtain
\begin{equation*}
\frac{\tau}{\sigma}\partial_{tt}\phi=\frac{\tau^2}{\sigma^2}\partial_{tt}\phi+\frac{d}{\sigma}\frac{\tau}{\sigma}\phi_t+\frac{\tau}{\sigma}w_t-\frac{\tau}{\sigma}a\phi_t-c\frac{\tau}{\sigma}\psi_t.
\end{equation*}
Then, multiplying both sides of the previous equation by $\frac{\tau}{\sigma}\partial_{tt}\phi|(e^{-s\alpha^\star}(\xi^\star)^{-6/2})_t|^2$ and integrating in $Q$, we get 
\begin{align*}
\frac{\tau^2}{\sigma^2}\iint_{Q}&\big|(e^{-s\alpha^\star}(\xi^\star)^{-6/2})_t\partial_{tt}\phi\big|^2\dx\dt\\
&= \frac{\tau^3}{\sigma^3}\iint_{Q}\big|(e^{-s\alpha^\star}(\xi^\star)^{-6/2})_t\partial_{tt}\phi\big|^2\dx\dt+d\frac{\tau^2}{\sigma^3}\iint_Q |(e^{-s\alpha^\star}(\xi^\star)^{-6/2})_t|^2 \partial_{tt}\phi\,\phi_t\dx\dt \\
&\quad +\frac{\tau^2}{\sigma^2}\iint_Q |(e^{-s\alpha^\star}(\xi^\star)^{-6/2})_t|^2 \partial_{tt}\phi\, w_t\dx\dt - a\frac{\tau^2}{\sigma^2}\iint_Q |(e^{-s\alpha^\star}(\xi^\star)^{-6/2})_t|^2 \partial_{tt}\phi\,\phi_t\dx\dt \\
&\quad -c\frac{\tau^2}{\sigma^2}\iint_Q |(e^{-s\alpha^\star}(\xi^\star)^{-6/2})_t|^2 \partial_{tt}\phi\,\psi_t\dx\dt.
\end{align*}
From Cauchy-Schwarz and Young inequalities, we readily deduce that for any $\delta>0$
\begin{align} \notag
\frac{\tau^2}{\sigma^2}\iint_{Q}&\big|(e^{-s\alpha^\star}(\xi^\star)^{-6/2})_t\partial_{tt}\phi\big|^2\dx\dt \\ \notag
&\leq 4\delta  \frac{\tau^2}{\sigma^2} \iint_{Q} \big| (e^{-s\alpha^\star} (\xi^\star)^{-6/2})_t \partial_{tt}\phi\big|^2\dx\dt +\frac{\tau^3}{\sigma^3}\iint_{Q}\big|(e^{-s\alpha^\star}(\xi^\star)^{-6/2})_t\partial_{tt}\phi\big|^2\dx\dt \\ \notag
&+C_{\delta}\frac{\tau^2}{\sigma^2}\iint_{Q}|(e^{-s\alpha^\star}(\xi^\star)^{-6/2})_t|^2|\phi_t|^2\dx\dt+C_{\delta}\frac{\tau^2}{\sigma^2}\iint_{Q}|(e^{-s\alpha^\star}(\xi^\star)^{-6/2})_t|^2|w_t|^2\dx\dt \\ \label{eq:est_der_phi}
&+C_{\delta}\frac{\tau^2}{\sigma^2}\iint_{Q}|(e^{-s\alpha^\star}(\xi^\star)^{-6/2})_t|^2|\psi_t|^2\dx\dt
\end{align}
where we have used that $\sigma>1$ to adjust the power of $\sigma$ in the term containing $\phi_t$. 

Using that $|(e^{-s\alpha^\star}(\xi^{\star})^{-6Ç/2})_t|\leq Cs^2(\xi^{\star})^{-1}e^{-s\alpha^\star}$, we can take $\delta>0$ small enough to obtain
\begin{align} \notag
\frac{\tau^2}{\sigma^2}s^{-6}\iint_{Q}&\big|(e^{-s\alpha^\star}(\xi^\star)^{-6/2})_t\partial_{tt}\phi\big|^2\dx\dt \\ \notag
\leq & \ C\frac{\tau^2}{\sigma^2}s^{-2}\iint_{Q}e^{-2s\alpha^\star}(\xi^\star)^{-2}|\phi_t|^2\dx\dt+C\frac{\tau^2}{\sigma^2}s^{-2}\iint_{Q}e^{-2s\alpha^\star}(\xi^\star)^{-2}|w_t|^2\dx\dt \\ \label{eq:deriv_phi_fin}
&+C\frac{\tau^2}{\sigma^2}s^{-2}\iint_{Q}e^{-2s\alpha^\star}(\xi^\star)^{-2}|\psi_t|^2\dx\dt.
\end{align}
Here, notice that we have also absorbed the second term in the right-hand side of \eqref{eq:est_der_phi} since $\tau<1$ and $\sigma>1$. 

Now, let us estimate the term containing $\partial_{tt}w$ in \eqref{eq:est_u_fin}. From the first equation of \eqref{eq:extended_adj}, we have
\begin{equation}\label{eq:w_tt_eq}
-\partial_{tt}w-\Delta w_t-aw_t=\frac{cb}{\sigma}\phi_t.
\end{equation}
Multiplying both sides of the above equation by $e^{-2s\alpha^\star}(\xi^\star)^{-6}\partial_{tt}w$ and using Cauchy-Schwarz and Young inequalities, we readily see that
\begin{align}\notag 
&s^{-6}\iint_{Q}e^{-2s\alpha^\star}(\xi^\star)^{-6}|\partial_{tt}w|^2\dx\dt\notag\\
&\leq Cs^{-6}\left(\iint_{Q}e^{-2s\alpha^\star}(\xi^\star)^{-6} \Delta w_t w_{tt} \dx\dt+\iint_{Q}e^{-2s\alpha^\star}(\xi^\star)^{-6} |w_t|^2\dx\dt\right.\notag \\ \label{eq:est_w_in}
&\quad \left. + \frac{1}{\sigma^2}\iint_{Q}e^{-2s\alpha^\star}(\xi^\star)^{-6}|\phi_t|^2\dx\dt \right).
\end{align}
Integrating by parts in space and then in time, we see that
\begin{align} \notag
s^{-6}\iint_{Q}e^{-2s\alpha^\star}(\xi^\star)^{-6}\Delta w_t w_{tt}\dx\dt&=\frac{1}{2}s^{-6}\iint_{Q}\big(e^{-2s\alpha^\star}(\xi^\star)^{-6}\big)_t |\nabla w_t|^2\dx\dt \\ \label{eq:integral_w}
& \leq Cs^{-4}\iint_{Q}e^{-2s\alpha^\star}(\xi^\star)^{-4}|\nabla w_t|^2\dx\dt
\end{align}
since $\big|(e^{-2s\alpha^\star}(\xi^\star)^{-6})_t\big|\leq C s^2e^{-2s\alpha^\star}\xi^4$. Using once again equation \eqref{eq:w_tt_eq}, we multiply both sides by $e^{-2s\alpha^\star}(\xi^{\star})^{-4} \partial_t w$ and integrate by parts in time and space, hence
\begin{align}\notag 
\int_{Q}e^{-2s\alpha^\star}(\xi^\star)^{-4}|\nabla w_t|^2\dx\dt=&-\frac{1}{2}\int_{Q}\big(e^{-2s\alpha^\star}(\xi^\star)^{-4}\big)_t|w_t|^2\dx\dt+a\int_{Q}e^{-2s\alpha^\star}(\xi^\star)^{-4}|w_t|^2\dx\dt \\ \label{est_nabla_w_t}
&+\frac{cb}{\sigma}\int_{Q}e^{-2s\alpha^\star}(\xi^\star)^{-4}\phi_t w_t.
\end{align}
Proceeding as before, it is not difficult to see that
\begin{align}\notag 
&s^4\iint_{Q}e^{-2s\alpha^\star}(\xi^\star)^{-4} |\nabla w_t|^2\dx\dt\\ \label{est_nabla_fin}
&\leq C\left(s^{-2}\iint_{Q}e^{-2s\alpha^\star}(\xi^\star)^{-2}|w_t|^2\dx\dt+\frac{1}{\sigma^2}s^{-4}\iint_{Q}e^{-2s\alpha^\star}(\xi^\star)^{-4}|\phi_t|^2\dx\dt\right)
\end{align}
Here, we have used that $|(e^{-2s\alpha^\star}(\xi^\star)^{-4})_t|\leq Cs^2e^{-2s\alpha^\star}(\xi^{\star})^{-2}$ and 
\begin{equation}\label{eq:prop_xi}
s^{-1}(\xi^{\star})^{-1}\leq C
\end{equation}
for some positive constant $C$ only depending on $\Omega$ and $\omega$ to gather all the terms containing $w_t$. 

\subsubsection*{Step 5. Conclusion}

We have now all the estimates for finishing our proof. Indeed, combining estimates \eqref{eq:est_init_phi_t}, \eqref{eq:est_phi_tt}, \eqref{eq:est_u_fin}, \eqref{eq:deriv_phi_fin}, \eqref{eq:est_w_in}, \eqref{eq:integral_w}, \eqref{est_nabla_w_t} and \eqref{est_nabla_fin}, we get
\begin{align*} \notag 
s^4&\iint_{\omega\times(0,T)}e^{-2s\widehat\alpha}(\widehat\xi)^4|\partial_t\phi|^2\dx\dt \\ \notag
 \leq &\  C\left(\frac{s^{14}}{2}\iint_{\omega\times(0,T)}e^{-4s\widehat{\alpha}+2s\alpha^\star}(\widehat\xi)^{14}|\phi^2|\dx\dt+\frac{s^4}{2}\iint_{\omega\times(0,T)}\big(e^{-2s\widehat \alpha}(\widehat{\xi})^4\big)_{tt}|\phi|^2\dx\dt\right)\\ \notag
&+C\left(s^{-2}\iint_{Q}e^{-2s\alpha^\star}(\xi^{\star})^{-2}|w_t|^2\dx\dt +\frac{1}{\sigma^2}s^{-4}\iint_{Q}e^{-2s\alpha^\star}(\xi^\star)^{-4}|\phi_t|^2\dx\dt  \right. \\ \notag
& \qquad\quad + \frac{\tau^2}{\sigma^2}s^{-2}\iint_{Q}e^{-2s\alpha^\star}(\xi^\star)^{-2}|\phi_t|^2\dx\dt+\frac{\tau^2}{\sigma^2}s^{-2}\iint_{Q}e^{-2s\alpha^\star}(\xi^\star)^{-2}|w_t|^2\dx\dt \\
&\qquad\quad  \left.+\frac{\tau^2}{\sigma^2}s^{-2}\iint_{Q}e^{-2s\alpha^\star}(\xi^\star)^{-2}|\psi_t|^2\dx\dt\right).
\end{align*}

Observe that we have estimated the local integral of $\partial_t\phi$ in terms of two local local terms of $\phi$ and several (global) lower order terms. Using \eqref{eq:prop_xi} and since $\tau<1$ and $\sigma>1$, we can further estimate
\begin{align} \notag 
s^4&\iint_{\omega\times(0,T)}e^{-2s\widehat\alpha}(\widehat\xi)^4|\partial_t\phi|^2\dx\dt \\ \notag
 \leq &\  C\left({s^{14}}\iint_{\omega\times(0,T)}e^{-4s\widehat{\alpha}+2s\alpha^\star}(\widehat\xi)^{14}|\phi^2|\dx\dt+{s^4}\iint_{\omega\times(0,T)}\big(e^{-2s\widehat \alpha}(\widehat{\xi})^4\big)_{tt}|\phi|^2\dx\dt\right)\\ \notag
&+C\left(s^{-2}\iint_{Q}e^{-2s\alpha^\star}(\xi^{\star})^{-2}|w_t|^2\dx\dt +s^{-2}\iint_{Q}e^{-2s\alpha^\star}(\xi^\star)^{-2}|\phi_t|^2\dx\dt  \right. \\ \notag
&\qquad\quad  \left.+\frac{\tau^2}{\sigma^2}s^{-2}\iint_{Q}e^{-2s\alpha^\star}(\xi^\star)^{-2}|\psi_t|^2\dx\dt\right).
\end{align}
Observe that here we have kept the factor $\tau^2/\sigma^2$ in the last term of the above inequality since the corresponding term in the left-hand side also have this, see \eqref{eq:car_psi_init}.

Using once again \eqref{eq:prop_xi} together with the fact that $\xi(x,t)\leq C\xi^\star(t)$ for some $C>0$ only depending on $\Omega$ and $\omega$, we can obtain the simplified expression
\begin{align} \notag 
s^4&\iint_{\omega\times(0,T)}e^{-2s\widehat\alpha}(\widehat\xi)^4|\partial_t\phi|^2\dx\dt \\ \notag
 \leq &\  C\left({s^{14}}\iint_{\omega\times(0,T)}e^{-4s\widehat{\alpha}+2s\alpha^\star}(\widehat\xi)^{14}|\phi^2|\dx\dt+{s^4}\iint_{\omega\times(0,T)}\big(e^{-2s\widehat \alpha}(\widehat{\xi})^4\big)_{tt}|\phi|^2\dx\dt\right)\\ \notag
&+C\left(T^2\iint_{Q}e^{-2s\alpha}\xi|w_t|^2\dx\dt +s^{-1}T^2\iint_{Q}e^{-2s\alpha}|\phi_t|^2\dx\dt  \right.  \left.+\frac{\tau^2}{\sigma^2}s^{-1}T^2\iint_{Q}e^{-2s\alpha}|\psi_t|^2\dx\dt\right)
\end{align}
where we have also used that $\xi^{-1}\leq C T^2$ and $e^{-2s\alpha^\star}\leq e^{-2s\alpha}$. Since  
\begin{equation*}
\big|\big(e^{-2s\widehat\alpha}(\widehat\xi)^{4}\big)_{tt}\big|\leq C s^4e^{-2s\widehat\alpha}(\widehat{\xi})^8
\end{equation*}
we obtain
\begin{align} 
\label{eq:final_estimate}s^4&\iint_{\omega\times(0,T)}e^{-2s\widehat\alpha}(\widehat\xi)^4|\partial_t\phi|^2\dx\dt \\ \notag
 \leq &\  C\left(s^{14}\iint_{\omega\times(0,T)}(e^{-4s\widehat{\alpha}+2s\alpha^\star}+e^{-2s\widehat\alpha})(\widehat\xi)^{14}|\phi^2|\dx\dt \right)\\ 
&+C\left(T^2\iint_{Q}e^{-2s\alpha}\xi|w_t|^2\dx\dt +s^{-1}T^2\iint_{Q}e^{-2s\alpha}|\phi_t|^2\dx\dt  \right.  \left.+\frac{\tau^2}{\sigma^2}s^{-1}T^2\iint_{Q}e^{-2s\alpha}|\psi_t|^2\dx\dt\right).\notag
\end{align}

With \eqref{eq:final_estimate}, we can estimate the right-hand side of \eqref{eq:car_sum_step3} and choosing the parameter $s\geq CT^2$, we can absorb the lower order terms to deduce
\begin{equation}\label{eq:car_sum_step5}
\begin{split}
I(w;1,2)+&I(\phi;1,1)+I(\psi;\tfrac{\tau}{\sigma},1) \\
&\leq C\left(s^{14}\iint_{\omega\times(0,T)}(e^{-4s\widehat\alpha+2s\alpha}+e^{-2s\widehat\alpha})(\widehat\xi)^{14}|\phi|^2\dx\dt \right).
\end{split}
\end{equation}
for every $s\geq s_2(T+T^2)$ where $s_2>0$ only depends on $\Omega$, $\omega$ and the coefficients $a,b,c,d$. The conclusion of the proof follows immediately by using the density of $C_0^\infty(\Omega)$ in $L^2(\Omega)$.
\end{proof}
\subsection{Proof of the uniform observability inequality for the adjoint system}
Once we have obtained the uniform Carleman estimate \eqref{car_k_libre}, the observability inequality \eqref{eq:Obs} follows immediately. 
\begin{proof}
The proof is standard and it is a consequence of \eqref{car_k_libre} and a dissipation estimate. In what follows, $C$ stands for a generic positive constant depending on the data of the problem, but uniform with respect to $\tau$ and $\sigma$. 

Multiplying system \eqref{eq:Syst_Linear_Adj} by $(\phi,\psi)$ and integrating by parts, we obtain
\begin{align}\label{eq:energy_phi}
&-\frac{1}{2}\frac{\d}{\dt}\int_{\Omega}|\phi|^2\dx+\int_{\Omega}|\nabla\phi|^2\dx=a\int_{\Omega}|\phi|^2\dx+c\int_{\Omega}\phi\psi\dx \\ \label{eq:energy_psi}
&-\tau\frac{1}{2}\frac{\d}{\dt}\int_{\Omega}|\psi|^2\dx+\sigma\int_{\Omega}|\nabla \psi|^2\dx-d\int_{\Omega}|\psi|^2\dx=b\int_{\Omega}\phi\psi\dx
\end{align}

Observe that since $d<0$, the last term in the left-hand side of \eqref{eq:energy_psi} is positive. Therefore, from Cauchy-Schwarz and Young inequalities we get from this equation that
\begin{align}\label{eq:energy_phibis}
&-\tau\frac{1}{2}\frac{\d}{\dt}\int_{\Omega}|\psi|^2\dx+\sigma\int_{\Omega}|\nabla \psi|^2\dx+\frac{1}{2}\int_{\Omega}|\psi|^2\dx\leq C\int_{\Omega}|\phi|^2\dx,
\end{align}
for some $C>0$ only depending on $d$ and $b$. Integrating  the above expression in $(t_1,t_2)\subseteq[0,T]$, we get
\begin{equation}\label{est:psi_unif}
\tau\|\phi(t_1,\cdot)\|_{L^2(\Omega)}^2+\iint_{\Omega\times(t_1,t_2)}|\psi|^2\dx\dt\leq C\left(\iint_{\Omega\times(t_1,t_2)}|\psi|^2\dx\dt+\tau\|\psi(t_2,\cdot)\|^2_{L^2(\Omega)}\right).
\end{equation}
where we have dropped the positive term containing the gradient of $\psi$. 

Repeating the analysis for \eqref{eq:energy_phi} and using estimate \eqref{est:psi_unif} together with Gronwall inequality yields%
\begin{equation*}
\|\phi(t_1,\cdot)\|^2_{L^2(\Omega)}+\tau\|\psi(t_1,\cdot)\|^2_{L^2(\Omega)}\leq C\left(\|\phi(t_2,\cdot)\|^2_{L^2(\Omega)}+\tau\|\psi(t_2,\cdot)\|^2_{L^2(\Omega)}\right)
\end{equation*}
for some $C>0$ uniform with respect to $\tau$ and $\sigma$. From the above inequality, we obtain
\begin{equation}\label{eq:est_energy_T4}
\|\phi(0,\cdot)\|^2_{L^2(\Omega)}+\tau\|\psi(0,\cdot)\|^2_{L^2(\Omega)}\leq \frac{2C}{T}\left(\iint_{\Omega\times(T/4,3T/4)} \left(|\phi|^2+|\psi|^2\right)\dx\dt\right)
\end{equation}
where we have used that $\tau<1$. Recalling our Carleman estimate, we readily have
\begin{equation}\label{eq:est_carleman_cut}
\begin{split}
s^4\iint_{\Omega\times(T/4,3T/4)}e^{-2s\alpha}\xi^4|\phi|^2&\dx\dt+s^4\iint_{\Omega\times(T/4,3T/4)}e^{-2s\alpha}\xi^4|\psi|^2\dx\dt  \\
&\leq Cs^{14}\iint_{\omega\times(0,T)}\big(e^{-2s\widehat\alpha}+e^{-4s\widehat \alpha+2s\alpha^\star}\big)(\widehat\xi)^{14}|\phi|^2\dx\dt.
\end{split}
\end{equation}

From here, it is standard to see that the Carleman weights in the left-hand side are bounded from below in the domain $\Omega\times(T/4,3T/4)$ (cf. \cite[Proposition 3.1]{gb_deT}). Indeed, 
\begin{equation}\label{eq:bound_below_w}
s^4e^{-2s\alpha}\xi^4\geq \frac{2^{16}}{3^4}s^4 T^{-8}\exp\left(-\frac{2^5 M_0 s}{3T^2}\right), \quad \forall (t,x)\in (T/4,3T/4)\times\overline\Omega,
\end{equation}
where we have denoted $M_0:=e^{2\lambda\|\eta^0\|_\infty}-1$ (see \eqref{weights} for recalling the definition of the weight $\alpha$).

For estimating the right-hand side, we recall that
\begin{equation}\label{eq:prop_weight}
e^{s\alpha^\star}\leq C e^{s(1+\epsilon)\widehat{\alpha}}
\end{equation}
for some $C>0$ only depending on $\epsilon$ and valid for every $\lambda>\lambda_0$ large enough, see the proof of \cite[Lemma 6.1]{montoya_deT}. 

Therefore, using \eqref{eq:prop_weight} first with $\epsilon=1$ and then with $\epsilon=1/3$, we deduce
\begin{equation*}
e^{-2s\widehat\alpha}+e^{-4s\widehat{\alpha}+2s\alpha^\star}\leq Ce^{-s\alpha^\star}\leq Ce^{-s\alpha}.
\end{equation*}
We also can prove that $s^{14}e^{-s\alpha}(\widehat\xi)^{14}\leq C$ for all $(t,x)\in (0,T)\times\overline \Omega$ by choosing $s$ sufficiently large. More precisely,
\begin{equation}\label{eq:bound_above_w}
s^{14}e^{-s\alpha}(\widehat \xi)^{14}\leq Cs^{14} 2^{28}T^{-28}e^{-\frac{4m_0s}{T^2}}\leq C\left(\frac{14}{em_0}\right)^{14}, \quad\forall (t,x)\in (0,T)\times\overline\Omega.
\end{equation}
if we choose $s\geq \left(\frac{14}{4m_0}\right)T^2$ and where we have defined $m_0:=e^{2\lambda\|\eta^0\|_\infty}-e^{\lambda\|\eta^0\|_\infty}$.

Therefore, combining \eqref{eq:est_carleman_cut} and \eqref{eq:bound_below_w}--\eqref{eq:bound_above_w}, we get
\begin{equation*}
\iint_{\Omega\times(T/4,3T/4)}|\phi|^2\dx\dt+\iint_{\Omega\times(T/4,3T/4)}|\psi|^2\dx\dt \leq  Ce^{\frac{Cs}{T^2}}\iint_{\omega\times(0,T)}|\phi|^2\dx\dt.
\end{equation*}
for every $s\geq s_3(T+T^2)$ where $s_3$ is a constant only depending on $\Omega,\omega$ and $a,b,c,d$. By setting $s=s_3(T+T^2)$ in the above estimate and combining it with \eqref{eq:est_energy_T4} we obtain the desired result. This ends the proof.
\end{proof}

\subsection{Proof of the uniform null-controllability result for the linearized system}
In this section, we prove the uniform controllability of the linear system \eqref{eq:Syst_Linearized}, i.e. we prove \Cref{prop:uniformNCLinear}. This will be done by employing the observability inequality \eqref{eq:Obs} and solving a suitable minimization problem.

\begin{proof}
The arguments presented here are by now classical and therefore we sketch them briefly. Let us consider, for any $\epsilon>0$, the following functional
\begin{equation}\label{eq:func_eps}
\begin{split}
J_{\epsilon}(\phi_T,\psi_T)=&\ \frac{1}{2}\iint_{\omega\times(0,T)}|\phi|^2\dx\dt+\int_{\Omega}\phi(0,\cdot)u_0\dx+\tau\int_{\Omega}\psi(0,\cdot)v_0\dx \\
&+\epsilon\left(\|\phi_T\|_{L^2(\Omega)}+\sqrt{\tau}\|\psi_T\|_{L^2(\Omega)}\right)
\end{split}
\end{equation}
where $(\phi,\psi)$ is the solution to \eqref{eq:Syst_Linear_Adj} associated to the initial datum $(\phi_T,\psi_T)\in [L^2(\Omega)]^2$. It is easy to prove that $J_\epsilon$ is continuous and strictly convex. Moreover, using our observability inequality \eqref{eq:Obs}, we get
\begin{align*}
J_\epsilon(\phi_T,\psi_T)\geq &\ \frac{1}{4}\iint_{\omega\times(0,T)}|\phi|^2\dx\dt-C_T\left(\|u_0\|^2_{L^2(\Omega)}+\tau\|v_0\|^2_{L^2(\Omega)}\right) \\
&+\epsilon\left(\|\phi_T\|_{L^2(\Omega)}+\sqrt{\tau}\|\psi_T\|_{L^2(\Omega)}\right).
\end{align*}
where $C_T$ is the constant appearing in \eqref{eq:Obs}. Consequently, $J_\epsilon$ is coercive in $[L^2(\Omega)]^2$ and the existence and uniqueness of a minimizer $(\phi_T^\epsilon,\psi_T^\epsilon)$, for each $\epsilon>0$, is guaranteed. 

Let $(\phi_T^\epsilon,\psi_T^\epsilon)$ be the unique minimizer of \eqref{eq:func_eps}. We assume that both $\phi_T^\epsilon,\psi_T^\epsilon \neq 0$, otherwise we can proceed as in \cite[Proof of Theorem 1.1]{FCG06}. We have
\begin{equation}\label{eq:deriv_J}
\left(J_\epsilon^\prime(\phi_T^\epsilon,\psi_T^\epsilon),(\phi_T,\psi_T)\right)=0, \quad\forall (\phi_T,\psi_T)\in [L^2(\Omega)]^2.
\end{equation}
We show that \eqref{eq:deriv_J} is in fact 
\begin{equation}\label{eq:opt_cond}
\begin{split}
\iint_{\omega\times(0,T)}&\phi^\epsilon\phi\dx\dt+\epsilon\left[\left(\frac{\phi_T^\epsilon}{\|\phi_T^\epsilon\|},\phi_T\right)_{L^2(\Omega)}+\sqrt{\tau}\left(\frac{\psi_T^\epsilon}{\|\psi_T^\epsilon\|},\psi_T\right)_{L^2(\Omega)}\right] \\
&+\int_{\Omega}\phi(0,\cdot)u_0\dx+\tau\int_{\Omega}\psi(0,\cdot)v_0\dx=0, \quad \forall (\phi_T,\psi_T)\in [L^2(\Omega)]^2
\end{split}
\end{equation}
where we $(\phi^\epsilon,\psi^\epsilon)$ stands for the solution to \eqref{eq:Syst_Linear_Adj} with initial datum $(\phi_T^\epsilon,\psi^\epsilon_T)$. 

Taking as a control $h=h^\epsilon=\phi^\epsilon1_{\omega\times(0,T)}$ in \eqref{eq:Syst_Linearized} and denoting the corresponding solution by $(u^\epsilon,v^\epsilon)$,  it can be shown by duality between systems \eqref{eq:Syst_Linearized} and \eqref{eq:Syst_Linear_Adj} and a comparison with \eqref{eq:opt_cond} that
\begin{align*}
\left(u^\epsilon(T,\cdot)-\epsilon\frac{\phi_T^\epsilon}{\|\phi_T^\epsilon\|_{L^2(\Omega)}},\phi_T\right)_{L^2(\Omega)}+\sqrt{\tau}\left(v^\epsilon(T,\cdot)-\epsilon\frac{\psi_T^\epsilon}{\|\psi_T^\epsilon\|_{L^2(\Omega)}},\psi_T\right)_{L^2(\Omega)}=0, \\
\forall (\phi_T,\psi_T)\in [L^2(\Omega)]^2
\end{align*}
whence
\begin{equation}
\label{eq:EstimationTSol}
\|u^\epsilon(T,\cdot)\|_{L^2(\Omega)}+\sqrt{\tau}\|v^\epsilon(T,\cdot)\|_{L^2(\Omega)}\leq \epsilon.
\end{equation}

Moreover, from the observability inequality and \eqref{eq:opt_cond} evaluated at the optimum $(\phi_T^\epsilon,\psi_T^\epsilon)$, we deduce
\begin{equation}
\label{eq:EstimationControl}
\|h^\epsilon\|_{L^2(\omega\times(0,T))}\leq \sqrt{C_T}\left(\|u_0\|^2_{L^2(\Omega)}+\tau\|v_0\|_{L^2(\Omega)}^2\right)^{1/2}
\end{equation}
where $C_T$ is the uniform constant coming from \eqref{eq:Obs}. 

In view of the inequalities \eqref{eq:EstimationTSol} and \eqref{eq:EstimationControl}, by taking limits as $\epsilon\to 0$ up to a subsequence, we deduce the existence of $h$ satisfying \eqref{eq:EstimationControl} steering the solution $(u,v)$ of \eqref{eq:Syst_Linearized} to $0$ at time $T$. This concludes the proof of \Cref{prop:uniformNCLinear}.
\end{proof}

\subsection{Source term method}
In this section, we adapt the source term method of \cite{LTT13} to our case. More precisely, from \Cref{prop:uniformNCLinear}, we have an estimate for the control cost in $L^2$ of the system \eqref{eq:Syst_Linearized}. Then we fix $M>0$ such that $C_T \leq M e^{M/T}$ where $C_T$ is defined in \eqref{eq:CTcout}. Let $q \in (1, \sqrt{2})$ and $p > q^{2}/(2-q^{2})$. We define the weights
\begin{equation}
\label{eq:rho0}
\rho_0(t) := M^{-p} \exp\left(- \frac{Mp}{(q-1)(T-t)}\right),
\end{equation}
\begin{equation}
\label{eq:rhoG}
\rho_{\mathcal{S}}(t) := M^{-1-p} \exp\left(-\frac{(1+p)q^{2}M}{(q-1)(T-t)}\right).
\end{equation}
\indent For $S \in L^2(0,T;L^2(\Omega))$, $h \in L^2(0,T;L^2(\Omega))$, $(u_0, v_0) \in L^2(\Omega)^2$, we introduce the following system:
\begin{equation}
\label{eq:Syst_Linear_S}
\left\{
\begin{array}{l l}
\partial_t u-  \Delta u = a u + b v + S +   h 1_{\omega} &\mathrm{in}\ (0,T)\times\Omega,\\
\tau \partial_t v -  \sigma \Delta v = c u + d v  &\mathrm{in}\ (0,T)\times\Omega,\\
u= \frac{\partial v}{\partial n}= 0\ &\mathrm{on}\ (0,T)\times\partial\Omega,\\
(u,v)(0,\cdot)=(u_0,v_0)& \mathrm{in}\ \Omega,
\end{array}
\right.
\end{equation}
Then, we define associated spaces for the source term, the state and the control
\begin{align}
& \mathcal{S} := \left\{S \in  L^2((0,T);L^2(\Omega))\ ;\ \frac{S}{\rho_{\mathcal{S}}} \in  L^2((0,T);L^2(\Omega))\right\},\label{eq:defpoidsS}\\
& \mathcal{Z}:= \left\{(u,v) \in L^2((0,T);L^2(\Omega)^2)\ ;\ \frac{(u,v)}{\rho_0} \in L^2((0,T);L^2(\Omega)^2)\right\},
\label{eq:defpoidsZ}\\
& \mathcal{H} := \left\{h \in L^2((0,T);L^2(\Omega))\ ;\ \frac{h}{\rho_{0}}  \in L^2((0,T);L^2(\Omega))\right\}.\label{eq:defpoidsh}
\end{align}
From the behaviours near $t=T$ of $\rho_{\mathcal{S}}$ and $\rho_0$, we deduce that each element of $\mathcal{S}$, $\mathcal{Z}$, $\mathcal{H}$ vanishes at $t=T$.
From an easy adaptation of \cite[Proposition 2.3]{LTT13}, we deduce the null-controllability for \eqref{eq:Syst_Linear_S}.
\begin{prop}\label{prop:SourceTerm}
For every $S \in \mathcal{S}$ and $(u_0, v_0) \in L^2(\Omega)^2$, there exists $h \in \mathcal{H}$, such that the solution $(u,v)$ of \eqref{eq:Syst_Linear_S} satisfies $(u,v) \in \mathcal{Z}$. Furthermore, $(u,v,h)$ satisfies the following estimate
\begin{align}
&\norme{u/\rho_0}_{C([0,T];L^2(\Omega))} + \sqrt{\tau} \norme{v/\rho_0}_{C([0,T];L^2(\Omega))} + \norme{h}_{\mathcal{H}}\notag\\
&\  \leq C_T \left( \norme{u_0}_{L^{2}(\Omega)} + \sqrt{\tau} \norme{v_0}_{L^{2}(\Omega)}  + \norme{S}_{\mathcal{S}}\right),\label{eq:EstimationLTT}
\end{align}
where $C_T$ is of the form \eqref{eq:CTcout}. In particular, since $\rho_0$ is a continuous function satisfying $\rho_0(T)=0$, the above relation \eqref{eq:EstimationLTT} yields $(u,v)(T,\cdot)=0$.
\end{prop}
For the sake of completeness, the proof of \Cref{prop:SourceTerm} is given in \Cref{sec:proofSTM}.\\
\indent The next proposition gives more information on the regularity of the controlled trajectory obtained in \Cref{prop:SourceTerm}. We define $\rho$ such that $\rho(T) = 0$, satisfying the inequalities
\begin{equation}
\label{eq:defrho}
\rho_0 \leq C \rho,\ \rho_{\mathcal{S}} \leq C \rho,\ |\rho'| \rho_0 \leq C \rho^2,\ \rho^2 \leq C \rho_{\mathcal{S}}.
\end{equation}
For instance, one can take
\begin{equation*}
\rho(t) =  \exp\left(- \frac{M\beta}{(q-1)(T-t)}\right),\ \text{with}\ \frac{(1+p)q^2}{2} < \beta < p.
\end{equation*}
\begin{prop}\label{prop:SourceTermReg}
For every $S \in \mathcal{S}$ and $(u_0, v_0) \in H_0^1(\Omega)\times H^1(\Omega)$, there exists an unique control $h$ of minimal norm in $\mathcal{H}$, such that the solution $(u,v)$ of \eqref{eq:Syst_Linear_S} satisfies 
\[ \frac{(u,v)}{\rho} \in H^1(0,T;L^2(\Omega)^2) \cap L^2(0,T;H^2(\Omega)^2) \cap C([0,T];H_0^1(\Omega)\times H^1(\Omega)).\] 
Moreover, the following estimate holds
\begin{align}
&\norme{u/\rho}_{H^1(0,T;L^2(\Omega)) \cap L^2(0,T;H^2(\Omega)) \cap C([0,T];H_0^1(\Omega))}\notag\\
&\notag + \sqrt{\tau} \norme{v/\rho}_{H^1(0,T;L^2(\Omega)) \cap C([0,T];H^1(\Omega))} + \norme{v/\rho}_{ L^2(0,T;H^2(\Omega))} + \norme{h}_{\mathcal{H}}\notag\\
&\  \leq C_T \left( \norme{u_0}_{H^{1}(\Omega)} + \sqrt{\tau} \norme{v_0}_{H^{1}(\Omega)}  + \norme{S}_{\mathcal{S}}\right),\label{eq:EstimationLTTReg}
\end{align}
where $C_T$ is of the form \eqref{eq:CTcout}.
\end{prop}
The proof of \Cref{prop:SourceTermReg} is a straightforward adaptation of \cite[Proposition 2.8, Proposition 2.9]{LTT13} and maximal regularity estimates given by \Cref{prop:EnergyEstimateMaxL2}.

\subsection{Fixed-point argument}
\label{sec:fixedpoint}
In this last part, we will give the proof of \Cref{th:mainresult2}.
\begin{proof}[Proof of \Cref{th:mainresult2}] In the following proof, $C=C_T$ will denote positive constants of the form \eqref{eq:CTcout} varying from line to line.\\
\indent Let $(u_0, v_0) \in H_0^1(\Omega)\times H^1(\Omega)$ such that 
\begin{equation}
\label{eq:tailledonnee}
\norme{(u_0,v_0)}_{H_0^1(\Omega)\times H^1(\Omega)} \leq r,
\end{equation}
which $r>0$ small enough that will be determined later, independent of $(\tau,\sigma)$. According to the previous subsection, for every $S \in \mathcal{S}$, there exists a (unique) control $h$ such that the corresponding trajectory $(u,v)$ of \eqref{eq:Syst_Linear_S} satisfies \eqref{eq:EstimationLTTReg}. It follows that, denoting
\begin{equation}
\label{eq:defballFixed}
\mathcal{S}_r := \{S \in \mathcal{S}\ ;\ \norme{S}_{\mathcal{S}} \leq r\},
\end{equation}
we can define an operator $\mathcal{N}$ acting on $\mathcal{S}_r$ by
\[ \mathcal{N}(S)(t) := g_1(u(t),v(t)) u(t)^2 + g_2(u(t)) u(t)v(t),\]
where $(u,v)$ is the trajectory of \eqref{eq:Syst_Linear_S} corresponding to the control input $h$. We recall from \eqref{eq:Hypf} that $g_1 \in W_0^{1,\infty}(\R^2), g_2 \in W_0^{1,\infty}(\R)$ but here we will only use that $g_1 \in W^{1,\infty}(\R^2), g_2 \in W^{1,\infty}(\R)$.\\
 \indent In order to obtain the conclusion of the proof of \Cref{th:mainresult2}, it suffices to check that, for $r>0$ small enough not depending on $(\tau, \sigma)$, $\mathcal{N}$ is a contraction mapping from $\mathcal{S}_r$ into itself.\\
\indent  \textit{Step 1: $\mathcal{S}_r$ is invariant for $\mathcal{N}$ provided that $r$ is small enough.} By using \eqref{eq:defrho} and the embedding $H^1(\Omega) \hookrightarrow L^4(\Omega)$ because the spatial dimension $N \leq 3$ (see \cite[Section 5.6]{Eva10}),
\begin{align*}
\norme{\frac{\mathcal{N}({S})}{\rho_{\mathcal{S}}}(t)}_{L^2(\Omega)} &\leq C \underbrace{\left|\frac{\rho^2(t)}{\rho_{\mathcal{S}}(t)}\right|}_{L^{\infty}(0,T)} \left(\norme{\frac{u(t)^2}{\rho(t)^2}}_{L^2(\Omega)} +\norme{\frac{u(t)v(t)}{\rho(t)^2}}_{L^2(\Omega)} \right)\\
& \leq C \left(\norme{\frac{u(t)}{\rho(t)}}_{L^4(\Omega)}^2 +\norme{\frac{u(t)}{\rho(t)}}_{L^4(\Omega)} \norme{\frac{v(t)}{\rho(t)}}_{L^4(\Omega)} \right )\\
& \leq C \left(\norme{\frac{u(t)}{\rho(t)}}_{H^1(\Omega)}^2 +\norme{\frac{u(t)}{\rho(t)}}_{H^1(\Omega)} \norme{\frac{v(t)}{\rho(t)}}_{H^1(\Omega)} \right ),
\end{align*}
then by integrating in time, and by using \eqref{eq:EstimationLTTReg}, \eqref{eq:tailledonnee} and \eqref{eq:defballFixed},
\begin{align*}
\norme{\frac{\mathcal{N}({\mathcal{S}})}{\rho_{\mathcal{S}}}}_{L^2(0,T;L^2(\Omega))} &\leq C \left(\norme{\frac{u}{\rho}}_{C([0,T];H_0^1(\Omega))}^2 +\norme{\frac{u}{\rho}}_{C([0,T];H_0^1(\Omega))} \norme{\frac{v}{\rho}}_{L^2(0,T;H^1(\Omega))} \right )\\
& \leq C \left( \norme{(u_0,v_0)}_{H_0^1(\Omega)\times H^1(\Omega)}^2 + \norme{S}_{\mathcal{S}}^2\right) \\
&\leq C r^2.
\end{align*}
Then, for $r>0$ small enough, $\mathcal{N}$ stabilises $\mathcal{S}_{r}$.\\
\indent \textit{Step 2: $\mathcal{N}$ is contracting for $r$ small enough.} We have 
\begin{align*}
| g_1(u_1,v_1)u_1^2 - g_1(u_2,v_2)u_2^2 |& =| g_1(u_1,v_1)u_1^2 - g_1(u_2,v_2)u_1^2 + g_1(u_2,v_2)u_1^2  - g_1(u_2,v_2)u_2^2 |\\
& \leq C |u_1^2| \left(|u_1 - u_2| + |v_1 - v_2| \right) + C |u_1 - u_2| |u_1 + u_2|,
\end{align*}
and 
\begin{align*}
| g_2(u_1)u_1v_1 - g_2(u_2)u_2 v_2 |& \leq | g_2(u_1)u_1v_1 - g_2(u_2)u_1 v_1 +  g_2(u_2)u_1 v_1 - g_2(u_2)u_2 v_2 |\\
& \leq C |u_1| |v_1| |u_1 - u_2|  + C |v_1| |u_1 - u_2| + C |u_2| |v_1-v_2|.
\end{align*}
Then, by using \eqref{eq:defrho}, Hölder estimates and the embeddings $H^1(\Omega) \hookrightarrow L^6(\Omega)$ because the spatial dimension $N \leq 3$ (see \cite[Section 5.6]{Eva10}), we deduce\small
\begin{align*}
&\norme{\frac{\mathcal{N}({S_1})-\mathcal{N}({S_2})}{\rho_{\mathcal{S}}}(t)}_{L^2}\\
& \leq C \Bigg(\norme{\frac{u_1(t)^2}{\rho(t)^2} \left(\frac{|u_1(t)-u_2(t)|}{\rho(t)} + \frac{|v_1(t)-v_2(t)|}{\rho(t)}\right)}_{L^2} + \norme{\left(\frac{|u_1(t)|}{\rho(t)} + \frac{|u_2(t)|}{\rho(t)} \right)\frac{|u_1(t)-u_2(t)|}{\rho(t)} }_{L^2}\\
& \quad+ \norme{\frac{|u_1(t)|}{\rho(t)}\frac{|v_1(t)|}{\rho(t)} \frac{|u_1(t)-u_2(t)|}{\rho(t)} }_{L^2} + \norme{\frac{|v_1(t)|}{\rho(t)}\frac{|u_1(t)-u_2(t)|}{\rho(t)}}_{L^2} + \norme{\frac{|u_2(t)|}{\rho(t)}\frac{|v_1(t)-v_2(t)|}{\rho(t)}}_{L^2} \Bigg)\\
& \leq C \Bigg( \norme{\frac{|u_1(t)|}{\rho(t)}}_{L^6}^2 \left(\norme{\frac{|u_1(t)-u_2(t)|}{\rho(t)}}_{L^6} + \norme{\frac{|v_1(t)-v_2(t)|}{\rho(t)}}_{L^6}\right) \\
& \quad + \left(\norme{\frac{|u_1(t)|}{\rho(t)}}_{L^4} +  \norme{\frac{|u_2(t)|}{\rho(t)}}_{L^4} \right) \norme{\frac{|u_1(t)-u_2(t)|}{\rho(t)}}_{L^4} + \norme{\frac{|u_1(t)|}{\rho(t)}}_{L^6} \norme{\frac{|v_1(t)|}{\rho(t)}}_{L^6} \norme{\frac{|u_1(t)-u_2(t)|}{\rho(t)}}_{L^6}\\
& \quad + \norme{\frac{|v_1(t)|}{\rho(t)}}_{L^4}\norme{\frac{|u_1(t)-u_2(t)|}{\rho(t)}}_{L^4} + \norme{\frac{|u_2(t)|}{\rho(t)}}_{L^4}\norme{\frac{|v_1(t)-v_2(t)|}{\rho(t)}}_{L^4}\Bigg),\\
& \leq C \Bigg( \norme{\frac{|u_1(t)|}{\rho(t)}}_{H^1}^2 \left(\norme{\frac{|u_1(t)-u_2(t)|}{\rho(t)}}_{H^1} + \norme{\frac{|v_1(t)-v_2(t)|}{\rho(t)}}_{H^1}\right) \\
& \quad + \left(\norme{\frac{|u_1(t)|}{\rho(t)}}_{H^1} +  \norme{\frac{|u_2(t)|}{\rho(t)}}_{H^1} \right) \norme{\frac{|u_1(t)-u_2(t)|}{\rho(t)}}_{H^1} + \norme{\frac{|u_1(t)|}{\rho(t)}}_{H^1} \norme{\frac{|v_1(t)|}{\rho(t)}}_{H^1} \norme{\frac{|u_1(t)-u_2(t)|}{\rho(t)}}_{H^1}\\
& \quad + \norme{\frac{|v_1(t)|}{\rho(t)}}_{H^1}\norme{\frac{|u_1(t)-u_2(t)|}{\rho(t)}}_{H^1} + \norme{\frac{|u_2(t)|}{\rho(t)}}_{H^1}\norme{\frac{|v_1(t)-v_2(t)|}{\rho(t)}}_{H^1}\Bigg),
\end{align*}
\normalsize
then by integrating in time, using \eqref{eq:EstimationLTTReg}, \eqref{eq:tailledonnee} and \eqref{eq:defballFixed},
\[ \norme{\frac{\mathcal{N}({S_1})-\mathcal{N}({S_2})}{\rho_{\mathcal{S}}}}_{L^2(0,T;L^2(\Omega))} \leq C (r^2 + r) \norme{S_1 - S_2}_{\mathcal{S}}.\]
Consequently, by taking $r$ sufficiently small, $\mathcal{N}$ is a contracting mapping on the closed ball $\mathcal{S}_r$. Therefore by the Banach fixed-point theorem, $\mathcal{N}$ has a unique fixed-point $S$. By denoting by $(u,v,h)$ the associated trajectory to $S$, we find that $(u,v,h)$ satisfies the system \eqref{eq:SystSL}, $(u,v)(T,\cdot) = 0$ then \eqref{eq:uvT} holds and $h$ satisfies the uniform bound \eqref{eq:boundh} thanks to \eqref{eq:EstimationLTTReg} and \eqref{eq:defballFixed}, which leads to the conclusion of \Cref{th:mainresult2}.
\end{proof}

\section{Local null-controllability of the semilinear heat equation}

\subsection{Asymptotic behaviour of the solution of the reaction-diffusion system as $(\tau,\sigma)\rightarrow (0,+\infty)$}
\label{sec:passlimit}
In this section, we prove that, roughly speaking, the solution $(u_{\tau,\sigma}, v_{\tau,\sigma}, h_{\tau,\sigma})$ of \eqref{eq:SystSL} converges to $(y, \fint_{\Omega} y, h)$, the solution of \eqref{eq:heatSL} as $(\tau,\sigma)\rightarrow (0,+\infty)$. More precisely, we have the following result, coming from \cite{HR00} (see also \cite{rodrigues}).
\begin{prop}
\label{prop:asymptoticSyst}
Let $(u_0,v_0) \in H_0^1(\Omega)\times H^1(\Omega)$. Assume that $h_{\tau,\sigma} \rightharpoonup h$ in $L^2((0,T)\times\omega)$ as $(\tau,\sigma) \rightarrow (0,+\infty)$. Then, up to a subsequence, the solution $(u_{\tau,\sigma}, v_{\tau,\sigma})$ of \eqref{eq:SystSL}, associated to the datum $(u_0,v_0)$ and the control $h_{\tau,\sigma}$, converges to $(y, \fint_{\Omega} y)$, where $y$ is the solution of \eqref{eq:heatSL}, associated to the datum $y_0:= u_0$ and the control $h$, as $(\tau,\sigma) \rightarrow (0,+\infty)$, in the following sense
\begin{align}
\partial_t u_{\tau,\sigma} &\rightharpoonup \partial_t y\ \text{in}\ L^2(Q_T),\notag\\
u_{\tau,\sigma}&\rightharpoonup^{*} y \ \text{in}\ L^{\infty}(0,T;H_0^1(\Omega)),\label{eq:convsyst}\\
 v_{\tau,\sigma} &\rightarrow \fint_{\Omega} y \ \text{in}\ L^{2}(0,T;H^1(\Omega))\notag.
\end{align}
\end{prop}

\begin{proof}
Because $(h_{\tau,\sigma})$ weakly converges in $L^2$ then it is bounded. By using the fact that $f$ is globally Lipschitz, thanks to the assumption \eqref{eq:Hypf}, and arguing as in the proofs of \Cref{prop:Energyestimates} and \Cref{prop:EnergyEstimateMaxL2}, we can show that there exists a constant $C>0$, independent of $\tau,\sigma$ such that
\begin{equation}
\label{eq:EstiuvSyst}
\norme{u}_{C([0,T];H_0^1(\Omega))}+ \norme{\partial_t u}_{L^2(0,T;L^2(\Omega))} + \norme{v}_{L^2(Q_T)} + \sqrt{\sigma} \norme{\nabla v}_{L^2(Q_T)} \leq C.
\end{equation}
\indent Then, we deduce from \eqref{eq:EstiuvSyst} that there exist $(u,v)$ such that after extracting subsequences, we have
\begin{equation}
\label{eq:convu}
u_{\tau \sigma} \rightharpoonup^{*} u\ \text{in}\ L^{\infty}(0,T;H_0^1(\Omega)),\ \partial_t u_{\tau,\sigma} \rightharpoonup \partial_t u\ \text{in}\ L^2(Q_T),
\end{equation}
\begin{equation}
\label{eq:convv}
v_{\tau \sigma} \rightharpoonup v\ \text{in}\ L^2(0,T;H^1(\Omega)),\ |\nabla v_{\tau\sigma}|_{L^2(Q_T)} \rightarrow 0.
\end{equation}
So, we deduce from \eqref{eq:convv} that $v = v(t)$ only depends on the time variable $t$. On the other hand, by integrating with respect to the spatial variable the second equation of \eqref{eq:SystSL} and by using the Neumann homogeneous boundary conditions, we obtain for every $t \in [0,T]$, 
\begin{equation}
\label{eq:IntTimeSpaceSecond}
\tau \int_{\Omega} v_{\tau \sigma}(t) - \tau \int_{\Omega} v_0 = \int_{0}^t \int_{\Omega} (u_{\tau \sigma} - v_{\tau\sigma}).
\end{equation}
Then, by setting $\xi_{\tau\sigma}(t) = \fint_{\Omega} v_{\tau\sigma}(t,x) \d x \in H^1(0,T)$, we have that $\xi_{\tau\sigma}$ solves
\begin{equation}
\label{eq:EDOxi}
\tau \xi_{\tau\sigma}' = \fint_{\Omega} u_{\tau\sigma} - \xi_{\tau\sigma}\ \text{in}\ (0,T).
\end{equation}
We have the following two lemmas.
\begin{lem} 
\label{lem:estimxi}
We have that $\xi_{\tau,\sigma}$ satisfies the following estimate
\begin{equation}
\label{eq:EstiLinftyxi}
\norme{ \xi_{\tau\sigma}}_{L^{\infty}(0,T)} \leq \left| \fint_{\Omega} v_0 \right| + \norme{ \fint_{\Omega} u_{\tau\sigma}}_{L^{\infty}(0,T)}.
\end{equation}
\end{lem}
\begin{proof}[Proof of \Cref{lem:estimxi}] The proof is the same as in \cite[Proof of Proposition 3.2]{HR00}. We introduce $\zeta(t) = \fint_{\Omega} u_{\tau\sigma}(t,x) d\xi$. Let $p>1$, we multiply \eqref{eq:EDOxi} by $|\xi|^{p-2} \xi$ then integrate in $(0,T)$ to obtain
\begin{equation}
\label{eq:ximultiply}
\frac{\tau}{p} |\xi(t)|^p - \frac{\tau}{p} |\xi(0)|^p + \int_0^T |\xi(s)|^p \d s = \int_0^T \zeta(s) |\xi(s)|^{p-2} \xi(s) \d s.
\end{equation}
We use Young's inequality with the conjugate exponents $p, p/(p-1)$, to bound the right hand side term of \eqref{eq:ximultiply} then obtain
$$\int_0^T \zeta(s) |\xi(s)|^{p-2} \xi(s) \d s \leq \frac{1}{p} \int_0^T |\zeta(s)|^p \d s  + \left(1 - \frac{1}{p}\right) \int_0^T |\xi(s)|^p \d s ,$$
so
\begin{equation*}
\frac{\tau}{p} |\xi(t)|^p + \frac{1}{p} \int_0^T |\xi(s)|^p \d s \leq \frac{1}{p} \int_0^T |\zeta(s)|^p \d s +  \frac{\tau}{p} |\xi(0)|^p.
\end{equation*}
By multiplying by $p$, then taking the power $1/p$, we get
\begin{equation*}
\norme{\xi}_{L^p((0,T))} \leq  \norme{\zeta}_{L^p((0,T))} +  \tau^{1/p} |\xi(0)|.
\end{equation*}
The results follows by sending $p \rightarrow + \infty$.
\end{proof}
\begin{lem}
\label{lem:EstiMeanValueuv}
We have the following estimate
\begin{equation}
\label{eq:EstiMeanValueuv}
\norme{\fint_{\Omega} u_{\tau\sigma} - \fint_{\Omega} v_{\tau\sigma}}_{L^2(0,T)} \leq C \sqrt{\tau}.
\end{equation}
\end{lem}
\begin{proof}[Proof of \Cref{lem:EstiMeanValueuv}]The proof borrow some arguments from \cite[Proof of Theorem 4.1]{HR00}. By using the identity \eqref{eq:EDOxi} and integrating by parts, we have
\begin{align*}
&\int_0^T \left(\fint_{\Omega} u_{\tau\sigma} - \xi_{\tau\sigma}\right)^2 dt = \tau \int_0^T \xi_{\tau\sigma}' \left(\fint_{\Omega} u_{\tau\sigma} - \xi_{\tau\sigma} \right)dt\\
& \leq  C_{0,T} \tau + \tau \norme{\frac{d}{dt} \fint_{\Omega} u_{\tau\sigma}}_{L^1(0,T)} 
\norme{\xi_{\tau\sigma}}_{L^{\infty}(0,T)},
\end{align*}
where the constant $C_{0,T}$ defined below is bounded uniformly in $(\tau,\sigma)$ 
\begin{equation*} 
C_{0,T} := \frac{1}{2}\left(\left|\fint_{\Omega} v_0\right|^2 - |\xi_{\tau,\sigma}(T)|^2\right) + \xi_{\tau,\sigma}(T) \fint u_{\tau,\sigma}(T) - \fint v_0 \fint u_0.
\end{equation*}
Moreover, by using the $L^2(Q_T)$ bound on $\partial_t u_{\tau,\sigma}$ given by \eqref{eq:EstiuvSyst} and the $L^{\infty}(0,T)$ bound on $\xi_{\tau,\sigma}$ given by \eqref{eq:EstiLinftyxi}, we deduce \eqref{eq:EstiMeanValueuv} which concludes the proof.\\
\indent Let us remark that in order to obtain the bound on $\norme{\frac{d}{dt} \fint_{\Omega} u_{\tau\sigma}}_{L^1(0,T)} $, we have used a different argument from \cite{HR00} thanks to a maximal regularity $L^2$ estimate, consequence of the regularity assumption of the initial data.
\end{proof}
We are now in position to finish the proof of \Cref{prop:asymptoticSyst}. By \eqref{eq:convu} and by using Aubin-Lions' lemma (see \cite[Section 8, Corollary 4]{Sim87}), we can assume, up to a subsequence that $u_{\tau,\sigma}$ strongly converges to $u$ in $L^2(Q_T)$ then
\begin{equation*}
\left|\fint_{\Omega} u_{\tau\sigma} - \fint_{\Omega} u\right| \rightarrow 0\ \text{in}\ L^2(0,T).
\end{equation*}
Therefore, from \eqref{eq:EstiMeanValueuv}, we have
$$\fint_{\Omega}v_{\tau\sigma} \rightarrow \fint_{\Omega} u\ \text{in}\ L^2(0,T).$$ 
Consequently, by Poincaré's inequality, and $|\nabla v_{\tau\sigma}| \rightarrow 0$ in $L^2(Q_T)$, we obtain 
\begin{equation}
\label{eq:IdentifyLimv}
v_{\tau\sigma} \rightarrow \fint_{\Omega} u\ \text{in}\ L^2(0,T;H^1(\Omega)).
\end{equation}
From \eqref{eq:convu} and \eqref{eq:IdentifyLimv}, we can pass to the limit in the first equation of \eqref{eq:SystSL} to obtain the conclusion of \Cref{prop:asymptoticSyst} where we have set $y := u$.
\end{proof}

\subsection{Proof of the local null-controllability result}
The goal of this section is to prove \Cref{th:mainresult1}, which will be an easy consequence of the uniform null-controllability result,see \Cref{th:mainresult2} for the reaction-diffusion system \eqref{eq:SystSL} and the asymptotic behaviour of these solutions, see \Cref{prop:asymptoticSyst}.
\begin{proof}[Proof of \Cref{th:mainresult1}]
Let $T>0$ be any positive time. Let $\delta>0$ be given by \Cref{th:mainresult2}. Let us define
\begin{equation}
\label{eq:deftildedelta}
\tilde{\delta} := \min ( \delta, |\Omega|^{1/2} \delta).
\end{equation} 
Let $y_0 \in H_0^1(\Omega)$ such that 
\begin{equation}
\label{eq:yoetit}
\norme{y_0}_{H_0^{1}(\Omega)} \leq \tilde{\delta}.
\end{equation}
Then, from \eqref{eq:deftildedelta} and \eqref{eq:yoetit}, one can check that
\begin{equation}
\norme{(y_0, \fint_{\Omega} y_0)}_{H_0^1(\Omega) \times H^1(\Omega)} \leq \delta.
\end{equation}
Then, by \Cref{th:mainresult2}, for every $(\tau,\sigma) \in (0,1)\times(1,+\infty)$, there exists $h_{\tau,\sigma}$ satisfying \eqref{eq:boundh} such that the solution $(u,v)_{\tau,\sigma}$, satisfies
\begin{equation}
\label{eq:SystSLProof}
\left\{
\begin{array}{l l}
\partial_t u_{\tau,\sigma}-  \Delta u_{\tau,\sigma} = f\left(u_{\tau,\sigma},v_{\tau,\sigma}\right) +  h_{\tau,\sigma} 1_{\omega} &\mathrm{in}\ (0,T)\times\Omega,\\
\tau \partial_t v_{\tau,\sigma} -  \sigma \Delta v_{\tau,\sigma} = u_{\tau,\sigma}-v_{\tau,\sigma}  &\mathrm{in}\ (0,T)\times\Omega,\\
u_{\tau,\sigma} = \frac{\partial v_{\tau,\sigma}}{\partial n}= 0,\ &\mathrm{on}\ (0,T)\times\partial\Omega,\\
(u,v)_{\tau,\sigma}(0,\cdot)=(y_0, \fint_{\Omega} y_0),\quad (u,v)_{\tau,\sigma}(T,\cdot) = 0& \mathrm{in}\ \Omega.
\end{array}
\right.
\end{equation} 
By using the fact that $(h)_{\tau,\sigma}$ is bounded in $L^2((0,T)\times\omega)$, we can extract a subsequence such that $(h)_{\tau,\sigma}$ weakly converges to a control $h$ in $L^2((0,T)\times\omega)$. Thus, we can pass to the limit in the system \eqref{eq:SystSLProof}. More precisely, by \Cref{prop:asymptoticSyst}, we deduce that up to a subsequence, the solution $(u_{\tau,\sigma}, v_{\tau,\sigma})$ of \eqref{eq:SystSL} converges to $(y, \fint_{\Omega} y)$, where $y$ is the solution of \eqref{eq:heatSL} associated to the datum $y_0= u_0$ and the control $h$, as $(\tau,\sigma) \rightarrow (0,+\infty)$, in the sense \eqref{eq:convsyst}. The fact that $y$ vanishes at time $t=T$, i.e. \eqref{eq:yT}, follows by passing to the limit in the last equation of \eqref{eq:SystSLProof} because $C([0,T];L^2(\Omega))$ is continuously embedded in $L^2(0,T;H_0^1(\Omega)) \times H^1(0,T;H^{-1}(\Omega))$, see \cite[Section 5.9.2, Theorem 3]{Eva10}. To conclude, $y$ satisfies \eqref{eq:heatSL} and \eqref{eq:yT}, which leads to the end of the proof of \Cref{th:mainresult1}.
\end{proof}

\section{Numerical results}

We devote this section to  illustrate numerically some of the results presented in the previous section. We adapt the well-known penalized Hilbert Uniqueness Method (HUM) as presented in \cite{boyer_HUM}. 

It is well-known that the functional \eqref{eq:func_eps} is not well suited for numerical tests since the terms in the $L^2$-norm are not differentiable at zero. Therefore, following the classical penalized HUM method, for any $\epsilon>0$, we will look for the control $h$ minimizing the primal functional given by
\begin{equation}\label{hum_func}
F_\epsilon(h):=\frac{1}{2}\iint_{\omega\times(0,T)}|h|^2\dx\dt+\frac{1}{2\epsilon}\left(\|u(T)\|^2_{L^2(\Omega)} + \tau\|v(T)\|_{L^2(\Omega)}^2\right),
\end{equation}
where $(u,v)$ is the solution to \eqref{eq:Syst_Linearized}. Since \eqref{hum_func} is continuous, coercive and strictly convex, the existence of a unique minimizer, that we denote by $h^\epsilon$, is guaranteed. 

Using Fenchel-Rockafellar theory (see, e.g., \cite{Ekeland}), we can identify an associated dual functional to \eqref{hum_func}: for any $\epsilon>0$ and $(\phi_T,\psi_T)\in L^2(\Omega)^2$, we introduce
\begin{align}\notag
J_\epsilon(\phi_T,\psi_T):=&\ \frac{1}{2}\iint_{\omega\times(0,T)}|\phi|^2\dx\dt+\frac{\varepsilon}{2}\left(\|\phi_T\|^2_{L^2(\Omega)}+\tau\|\psi_T\|_{L^2(\Omega)}^2\right)\\ \label{hum_dual}
&+\int_{\Omega}\phi(0)u_0\dx+\tau\int_{\Omega}\psi(0)v_0\dx,
\end{align}
where $(\phi,\psi)$ is the solution to \eqref{eq:Syst_Linear_Adj} associated to the initial data $(\varphi_T,\psi_T)$. It is not difficult to see that \eqref{hum_dual} is continuous and strictly convex. Moreover, thanks to the observability inequality \eqref{eq:Obs}, we can prove that \eqref{hum_dual} is coercive in $L^2(\Omega)^2$, and hence the existence and uniqueness of a minimizer $(\phi_{T}^\epsilon,\psi_{T}^\epsilon)$ is also guaranteed.

Using well-known arguments (see, e.g., \cite[Proposition 1.5]{boyer_HUM}), it can be readily seen that the minimizers $h^\epsilon$ and $(\phi_{T}^\epsilon,\psi_{T}^{\epsilon})$ are related through the formulas
\begin{equation}\label{dual_id}
h^\varepsilon=\phi^{\varepsilon}|_{\omega}, \quad u^\varepsilon(T)=-\epsilon\phi_{T}^{\epsilon}, \quad v^\epsilon(T)=-\epsilon \psi_{T}^\epsilon
\end{equation}
where $\phi^\epsilon$ is taken from $(\phi^{\epsilon}, \psi^{\epsilon})$ solution to \eqref{eq:Syst_Linear_Adj} with initial data $(\phi_{T}^{\epsilon},\psi_{T}^{\epsilon})$ and where $(u^\epsilon,v^\epsilon)$ stands for the solution to \eqref{eq:Syst_Linearized} with control $h_\epsilon$. 

The following result allows us to relate the null controllability property for system \eqref{eq:Syst_Linearized} with the behaviour of these minimizers with respect to $\epsilon$. In more detail, we have
\begin{prop}\label{prop_cont_hum}
System \eqref{eq:Syst_Linearized} is null controllable if and only if 
\begin{equation}\label{m_limit}
\mathcal M^2:=2\sup_{\epsilon>0}\left(\inf_{L^2(\omega\times(0,T))}F_\epsilon\right)<+\infty.
\end{equation}
In this case, we have, 
\begin{equation}\label{eq:cost_M}
\|h^\epsilon\|_{L^2(\omega\times(0,T))}\leq \mathcal M \quad \text{and}\quad  \left(\|u^\epsilon(T)\|^2_{L^2(\Omega)}+\tau\|v^{\epsilon}(T)\|^2_{L^2(\Omega)}\right)^{1/2}\leq \mathcal M\sqrt\epsilon.
\end{equation}
\end{prop}

\begin{rmk} Some remarks are in order:
\begin{itemize}
\item The proof of Proposition \ref{prop_cont_hum} follows from a  straightforward adaptation of \cite[Theorem 1.7]{boyer_HUM}. We shall mention that in such procedure we assume that $\tau$ and $\sigma$ in \eqref{eq:Syst_Linearized} are fixed and no other conditions on the coupling coefficients are given. Actually, this result does not tell anything about the uniformity of the constant $\mathcal M$ with respect to the parameters $\tau$ and $\sigma$. Notwithstanding, we will use the above result to  illustrate the controllability at the numerical level and then we will use our computational code to test the uniformity of the constant $\mathcal M$ with respect to the parameters $\tau$ and $\sigma$. 
\item The control $h_\varepsilon$ can be computed directly by minimizing \eqref{hum_func}, but the space where the minimization is carried out depends on the time variable. From a practical point of view, it is easier to minimize the dual functional \eqref{hum_dual} and then use the identities \eqref{dual_id} to study the behaviour of the minimizer with respect to $\epsilon$.
\end{itemize}
\end{rmk}

Since the functional \eqref{hum_dual} is convex, quadratic and coercive, the conjugate gradient algorithm is a natural and simple choice to minimize it. A straightforward computation yields to
\begin{equation}\label{grad_prob}
\nabla J_\epsilon(\phi_T,\psi_T)=\Lambda(\phi_T,\psi_T)+\epsilon(\phi_T, \tau \psi_T)+(\overline{u}(T),\tau\overline{v}(T))
\end{equation}
with the Gramiam operator $\Lambda$ defined as follows 
\begin{align*}
\Lambda: L^2(\Omega)^2  \quad &\to \quad\  L^2(\Omega)^2, \\
(\phi_T,\psi_T)\ \quad &\mapsto \quad (w(T),\tau z(T)),
\end{align*}
where $(w(T),\tau z(T))$ can be found from the solution to the forward-backward systems 
\begin{equation}\label{gram_back}
\begin{cases}
-\phi_t-\Delta \phi = a\, \phi+c\, \psi & \text{in } (0,T)\times\Omega, \\
-\tau\psi_t-\sigma\Delta \psi=b\,\phi+d\,\psi &  \text{in }(0,T)\times\Omega, \\
\D \phi = \frac{\partial \psi}{\partial n}=0 &\text{on } (0,T)\times\partial\Omega, \\
(\phi,\psi)(T,\cdot)=(\phi_T,\psi_T)& \text{in } \Omega,
\end{cases}
\end{equation}
and
\begin{equation}\label{gram_forw}
\begin{cases}
w_t-\Delta w=a\, w+b\, z+ \phi 1_{\omega} & \text{in } (0,T)\times\Omega, \\
\tau{z_t}-\sigma\Delta z=c\,w+d\,z &  \text{in }(0,T)\times\Omega, \\
\D w=\frac{\partial z}{\partial n}=0 &\text{on } (0,T)\times\partial \Omega, \\
(w,z)(0,\cdot)=(0,0)& \text{in } \Omega,
\end{cases}
\end{equation}
and where the pair $(\overline u(T),\overline v(T))$ can be obtained from the \emph{free} solution to \eqref{eq:Syst_Linearized}, namely, the solution at time $T$ with control $h\equiv 0$ and initial data $(u_0,v_0)$. 

In this way, the minimizer $h_\epsilon$ can be obtained by as follows: for given $\epsilon>0$, we compute $(\phi_{T}^{\epsilon},\psi_{T}^{\epsilon})$, the solution to the linear problem 
\begin{equation}\label{sol_grad}
(\Lambda+\epsilon I)(\phi_T,\psi_T)=-\left(\overline{u}(T),\tau \overline v(T)\right),
\end{equation}
we compute the corresponding adjoint state with this initial data and finally we use the first formula in \eqref{dual_id} to obtain the control. Then, according to Proposition \ref{prop_cont_hum}, the expected controllability result can be tested by analyzing the behaviour of the involved quantities with respect to the parameter $\epsilon$. 

For the numerical tests, we consider the 1-$d$ spatial domain $\Omega=(0,1)$ and choose a time horizon of $T=0.1$, since we are mostly interested in the small-time controllability and also due to the fast diffusion effect of the second component of system \eqref{eq:Syst_Linearized}. 

Systems \eqref{eq:Syst_Linearized} and \eqref{gram_back}-\eqref{gram_forw} are discretized in the time variable by using the standard implicit Euler scheme with a uniform time step given by $\delta t=T/M$ where $M$ is the number of steps on the mesh. The PDEs are discretized in space by a standard finite-difference scheme (adapted to the corresponding boundary condition) with a constant discretization step of size $h=1/(N+1)$, where $N$ is the number of steps. 

We denote by $E_h$, $U_h$ and $L^2_{\delta t}(0,T;U_h)$ the discrete spaces associated to $L^2(\Omega)\times L^2(\Omega)$, $L^2(\omega)$ and $L^2((0,T)\times\omega)$, respectively. We denote by $F_{\epsilon}^{h,\delta t}$ the discretization of the functional $F_\epsilon$ and by $(u^{\epsilon,h,\delta t},v^{\epsilon,h, \delta t},h^{\epsilon,h,\delta t})$ the solution to the corresponding minimization problem. 

As usual in this context, to connect the discretization scheme to the control problem, we  use the penalization parameter $\epsilon=\phi(h)=h^{4}$. This choice is consistent with the order of approximation of the finite difference scheme. We refer the reader to \cite[Section 4]{boyer_HUM} for a more detailed discussion on the selection of the function $\phi(h)$ in the context of the null-controllability of some parabolic problems.

\subsection{Numerical controllability for fixed $\tau$ and $\sigma$}

In this part, we are interested in illustrating the controllability at the numerical level of a 1-$d$ version of the fast diffusion system \eqref{eq:Syst_Linearized}. To this end, consider the system given by
\begin{equation}\label{eq:sys_example}
\begin{cases}
u_t- u_{xx}=a\,u+b\, v+h1_\omega & \text{in } (0,T)\times(0,1), \\
\tau{v_t}-\sigma v_{xx}=c\,u+d\,v &  \text{in }(0,T)\times(0,1), \\
\D u=v_x=0 &\text{on } (0,T)\times\{0,1\}, \\
(u,v)(0,\cdot)=(u_0,v_0)& \text{in } (0,1).
\end{cases}
\end{equation}
As long as $c\neq 0$, a simple adaptation of \cite[Theorem 1.2]{gb_deT} allows us to establish the null controllability of \eqref{eq:sys_example} regardless the choice of parameters $\tau,\sigma>0$ and the other coupling coefficients $a,b,d$ (however the proof says nothing about the uniformity with respect to $\tau$ and $\sigma$). We illustrate below this fact at the numerical level with the aid of Proposition \ref{prop_cont_hum}.

Using our computational tool, we begin by obtaining the solution for system \eqref{eq:sys_example} without any control. We consider the set of parameters 
\begin{equation}\label{eq:param_test}
\begin{gathered}
a=2, \quad b=-\frac{1}{2} \quad c=\frac{11}{2} \\
u_0(x)=\sin(\pi x), \quad v_0(x)=1_{(0.2,0.7)}(x),
\end{gathered}
\end{equation}
and 
\begin{equation*}
\tau=0.5, \quad \sigma=2
\end{equation*}
and plot the time evolution of the uncontrolled system in Figures \ref{fig:libre_dn} and \ref{fig:fd_libre_dp} for two different parameters $d$. We observe from both figures that the solution over time of the component $u$ of the system is damped over time, but the behaviour of $v$ differs drastically depending the sign of the coefficient $d$: while for negative $d<0$ the solution is damped over time, for $d>0$ its size increases.

%
%
%
%
%
%

\begin{figure}[htbp]
  \centering
  \subfloat[The state][The state $(t,x)\mapsto u(t,x)$]{
\includegraphics{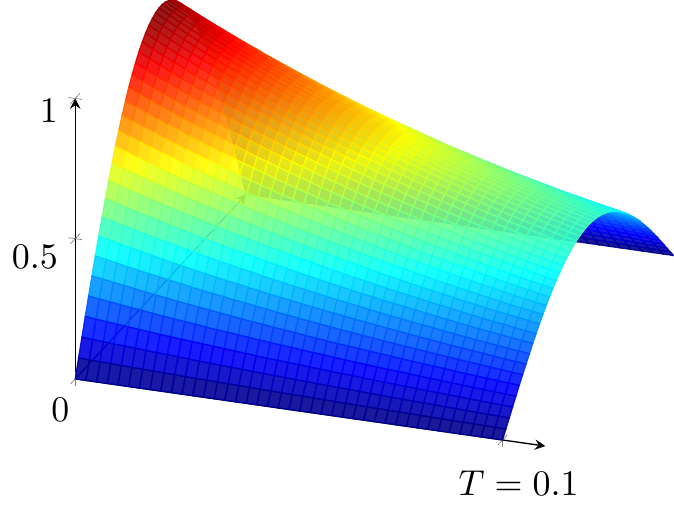}
} \;
\subfloat[The state][The state $(t,x)\mapsto v(t,x)$]{
\includegraphics{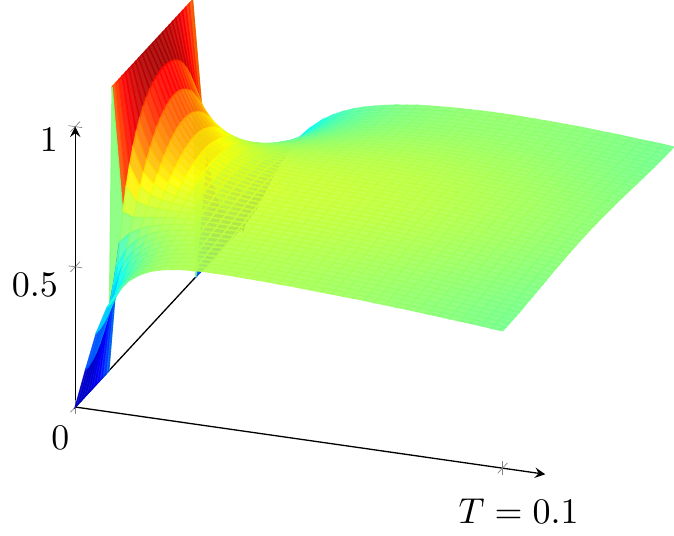}
}
\caption{Evolution in time of the uncontrolled fast-diffusion system for a parameter $d=-9/2$.}
\label{fig:libre_dn}
\end{figure}

%
%
%
%
%
%

\begin{figure}
\centering
  \subfloat[The state][The state $(t,x)\mapsto u(t,x)$]{
\includegraphics{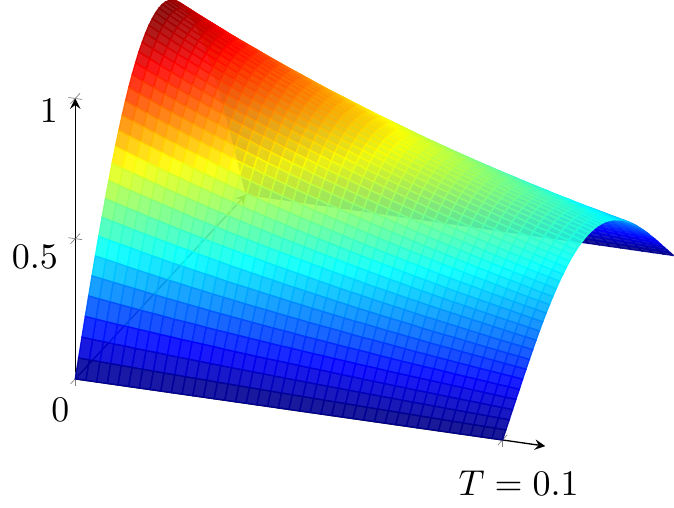}
} \;
\subfloat[The state][The state $(t,x)\mapsto v(t,x)$]{
\includegraphics{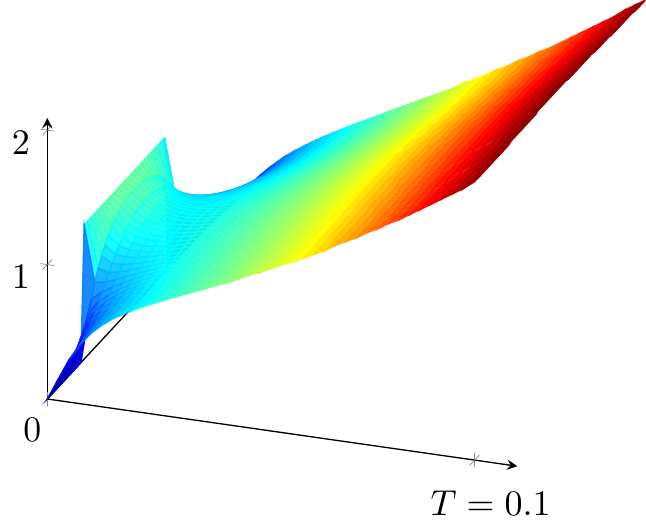}
}
\caption{Evolution in time of the uncontrolled fast-diffusion system for  $d=5$.}
\label{fig:fd_libre_dp}
\end{figure}

In Figures \ref{fig:ctr_dn} and \ref{fig:fd_ctrn_dp}, we plot the solutions $(u,v)$ obtained with the HUM control computed by the algorithm described in \eqref{grad_prob}-\eqref{sol_grad}. Once again, we test for the the parameters $d=-9/2$ and $d=9/2$ and in both cases we observe that, due to the action of the control, both of the components of the state move towards zero at the prescribed time $T=0.1$. For these experiments, we have chosen the control interval $\omega=(0.3,0.8)$.

%
%
%
%
%
%

\begin{figure}[htbp]
  \centering
  \subfloat[The state][The state $(t,x)\mapsto u(t,x)$]{
\includegraphics{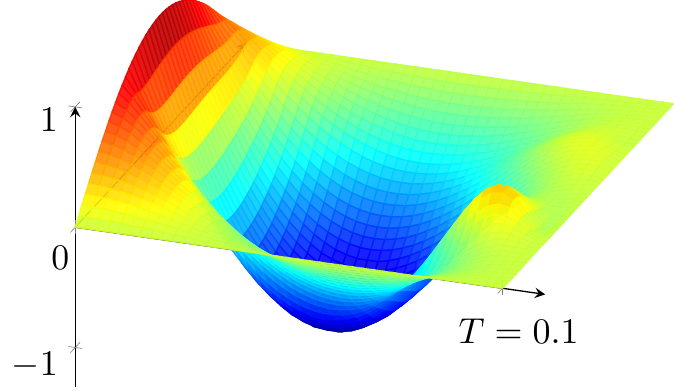}
} \;
\subfloat[The state][The state $(t,x)\mapsto v(t,x)$]{
\includegraphics{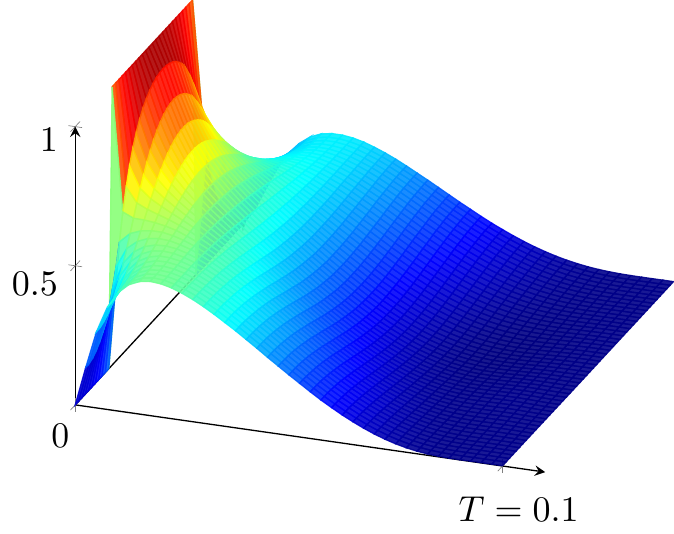}
}
\caption{Evolution in time of the controlled fast-diffusion system for a parameter $d=-9/2$.}
\label{fig:ctr_dn}
\end{figure}

%
%
%
%
%
%

\begin{figure}
\centering
  \subfloat[The state][The state $(t,x)\mapsto u(t,x)$]{
\includegraphics{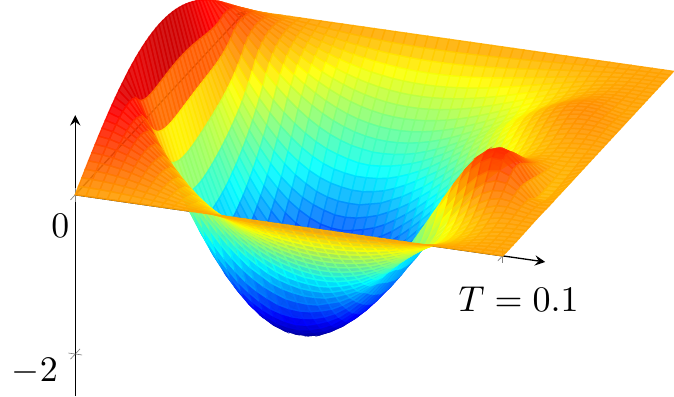}
} \;
\subfloat[The state][The state $(t,x)\mapsto v(t,x)$]{
\includegraphics{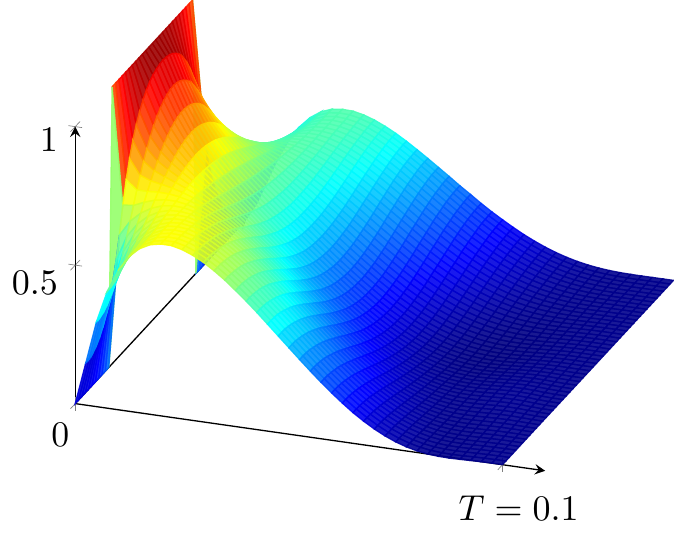}}
\caption{Evolution in time of the controlled fast-diffusion system for  $d=5$.}
\label{fig:fd_ctrn_dp}
\end{figure}

As far as the asymptotic behavior of the method is concerned, we present in Figure \ref{conv_h} the behavior of various quantities of interest when the mesh size goes to 0 for the corresponding cases $d<0$ and $d>0$. In more detail, in both cases we observe that the control cost $\|h_{\phi(h)}^{h,\delta t}\|_{L^2_{\delta t}(0,T;U_h)}$ (\ref{cost}) as well as the optimal energy $\inf F_{\phi(h)}^{h,\delta t}$ (\ref{energie}) remain bounded as the mesh size $h \to 0$. Also, we see that the norm of the controlled state $\|(u^{h,\delta t}(T),v^{h,\delta t}(T)\|_{E_h}$ (\ref{cible}) behaves like $\sim C\sqrt{\phi(h)}=Ch^2$ as predicted by Proposition \ref{prop_cont_hum}.

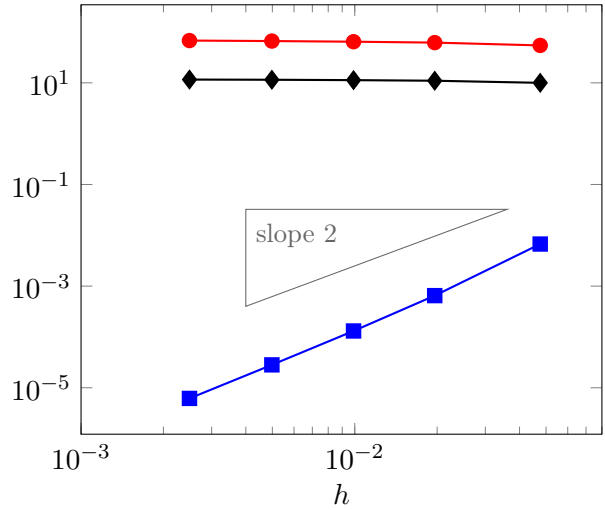
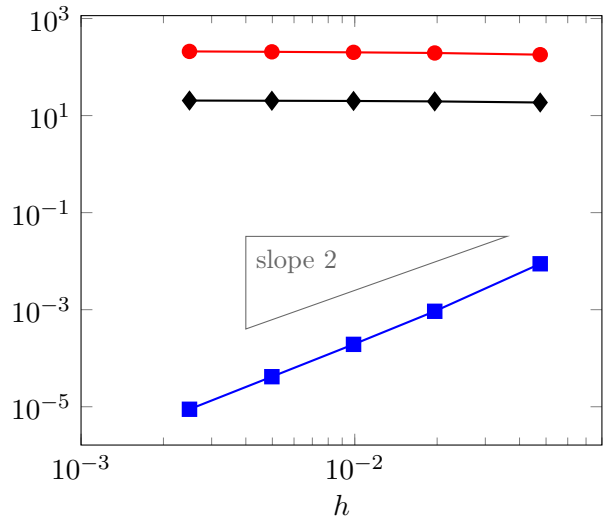
\begin{figure}[h!]
  \centering
    \ref{named} \\
  \subfloat[The case $d=-5$]{
  \begin{tikzpicture}
  \begin{loglogaxis}[erreurs,xmin=0.001,
    xmax=0.08, legend columns=-1,legend entries={Cost of the control\;\; ,Size of target\;\;, Optimal energy\;\;},legend to name=named]
 
    \pgfplotstableread[ignore chars={|},skip first n=2]{./results_nonlocal_04-Dec-2019_14h5.org}\resultats

    \addplot[cout, thick,black,mark options={fill=black}] table[x=dx,y=Nv] \resultats; \label{cost}
    \addplot[cible, thick] table[x=dx,y=NyT] \resultats; \label{cible}
    \addplot[energie, thick] table[x=dx,y=Inf_eps(F_eps)] \resultats; \label{energie}
    \draw [pente]  (axis cs: 0.004,0.04e-2) -- ++ (axis cs: 1, {9^(2)}) -- ++ (axis cs: 9, 1) -- cycle;
    \node at (axis cs:0.004,1e-2) [right,pente] {\small slope $2$};
    
  \end{loglogaxis}
  
\end{tikzpicture}
} \\
 \subfloat[The case $d=5$]{
  \begin{tikzpicture}
  \begin{loglogaxis}[erreurs,xmin=0.001,
    xmax=0.08]
 
    \pgfplotstableread[ignore chars={|},skip first n=2]{./results_nonlocal_04-Dec-2019_17h14.org}\resultats

    \addplot[cout, thick,black,mark options={fill=black}] table[x=dx,y=Nv] \resultats; \label{cost}
    \addplot[cible, thick] table[x=dx,y=NyT] \resultats; \label{cible}
    \addplot[energie, thick] table[x=dx,y=Inf_eps(F_eps)] \resultats; \label{energie}
    \draw [pente]  (axis cs: 0.004,0.04e-2) -- ++ (axis cs: 1, {9^(2)}) -- ++ (axis cs: 9, 1) -- cycle;
    \node at (axis cs:0.004,1e-2) [right,pente] {\small slope $2$};
    
  \end{loglogaxis}
  
\end{tikzpicture}
} 
\caption{Convergence properties of the method for the control of the fast diffusion system for fixed $\tau$ and $\sigma$.}
\label{conv_h}
\end{figure}



\subsection{Uniformity with respect to the parameters $\tau$ and $\sigma$}

As we have mentioned, Proposition \ref{prop_cont_hum} is valid for fixed values of $\tau$ and $\sigma$, and no other information on the uniformity with respect to these parameters can be obtained. Nevertheless, with our computational tool at hand, we can play with the values of such parameters and observe numerically the size of the constant $\mathcal M$ involved in \eqref{m_limit}. Then, by means of \eqref{eq:cost_M}, we can discuss on the uniformity of the size of the control. 

In Figure \ref{conv_h} we have shown the asymptotic behavior of the numerical method and from there we can see that the method gives a good approximation of the control independently of the sign of $d$. Here, to focus on the discussion of the uniformity with respect to $\tau$ and $\sigma$, we will set $N$ and $M$ (the number of mesh points in space and time) to a fixed value and will vary the values of the limiting parameters to make a discussion. 

\subsubsection{The case $d<0$}

To observe the uniformity of the constant $\mathcal M$ we will consider a sequence of parameters $(\tau,\sigma)$ going simultaneously to $(0,+\infty)$ and run our computational code for each given pair. In order to simplify the computations and the presentation of the results, we will make the choice of $\sigma=\frac{2}{\tau}$. For simulation purposes, we consider the set of data \eqref{eq:param_test} together with $d=-5$ and to ensure a good approximation of the control we set $N=400$ and $M=2000$. 

In figure \ref{fig:M_dn}, we plot the sequence of $\tau$ against the computed value of the constant $\mathcal M$. We can see that as $\tau$ decreases (which also translates into considering a bigger $\sigma$ in each step) the constant $\mathcal M$ converges to a fixed quantity. This can be explained by looking at the size of the free solution of \eqref{eq:sys_example}. Indeed, in Figure \ref{fig:size_free_dn} we see that after a certain threshold the norm of the solution at time $T$ of the uncontrolled system remains practically unchanged which translates into a control effort which is independent of the parameters $\tau$ and $\sigma$. Of course this behavior is not surprising since we already have pointed out in Remark \ref{rmk:unif_c_d} that the condition $d<0$ is necessary to obtain a uniform energy estimate for the solutions to \eqref{eq:sys_example}.

\begin{figure}
\centering
 \subfloat[Boundedness of the constant $\mathcal M$]{\label{fig:M_dn}
\includegraphics{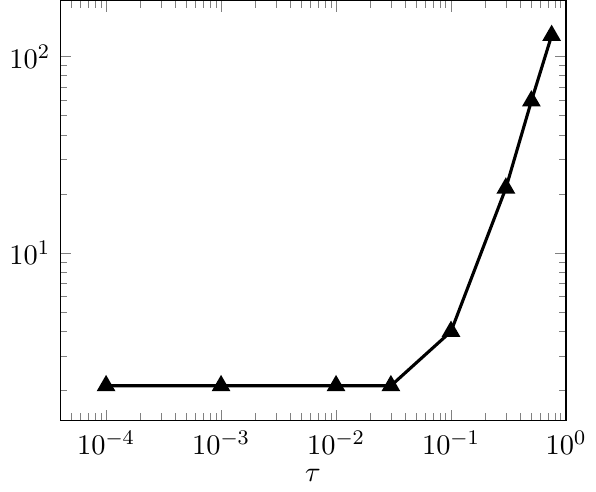}
} \qquad
\subfloat[Size of the free solution]{\label{fig:size_free_dn}
\includegraphics{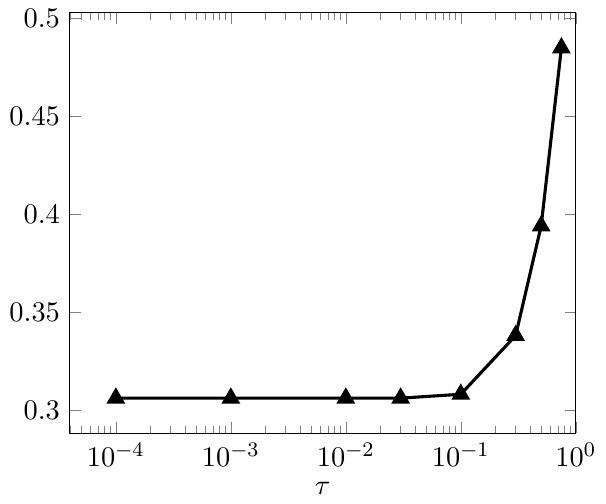}
}
\caption{Uniformity with respect to the parameters $(\tau,\sigma)$ for the case $d<0$.}
\label{fig:fd_libre_dp}
\end{figure}

\subsubsection{The case $d>0$}

As seen in Figure \ref{fig:fd_libre_dp}, the component $v$ of the solution of system \eqref{eq:sys_example} is not damped in time when we take $d=9/2$ and in fact this behavior can be observed for any coefficient $d>0$. By rescaling, taking the parameter $\tau$ in front of the time derivative has the same effect as extending the time interval where the equation is posed, therefore, one should be careful while simulating the behavior of a unstable system as $\tau\to 0$ since the size of the solution can grow very fast. 

Actually, the implicit Euler scheme that we are implementing is somehow impractical for computing the solution of unstable systems and, in this case, we can \textit{roughly} estimate that just for ensuring the stability of the numerical scheme, we need to fulfill the following condition on the discretization variables 
\begin{equation}\label{eq:formula_dt}
\frac{|d|\delta t}{\tau^2}\leq h^2
\end{equation}

For our particular example, taking $d=-4.5$, $h=1/25$ and $\tau=0.03$, formula \eqref{eq:formula_dt} implies that $\delta t=3.2\times 10^{-7}$ which means that we have to take $M=312500$ points in the time mesh. This condition is still ``manageable'' at the computational level and allows us to obtain some valuable information on the constant $\mathcal M$. Using the same data as in the previous case, in Figure \ref{fig:not_bound_M} we show the size of the computed constant $\mathcal M$ for a decreasing sequence of $\tau$. We can see that after some value the size of the constant $\mathcal M$ starts to increase monotonically. This behavior is obviously related with the norm of the uncontrolled solution, indeed, in Figure \ref{fig:grow_norm} we see that as $\tau$ decreases the size of the free solution increases, making harder the process of controlling the system. Therefore, it is unreasonable to expect for the control $h$ to be uniformly bounded.

\begin{figure}[htbp]
  \centering
  
  \subfloat[Constant $\mathcal M$]{\label{fig:not_bound_M}
\includegraphics{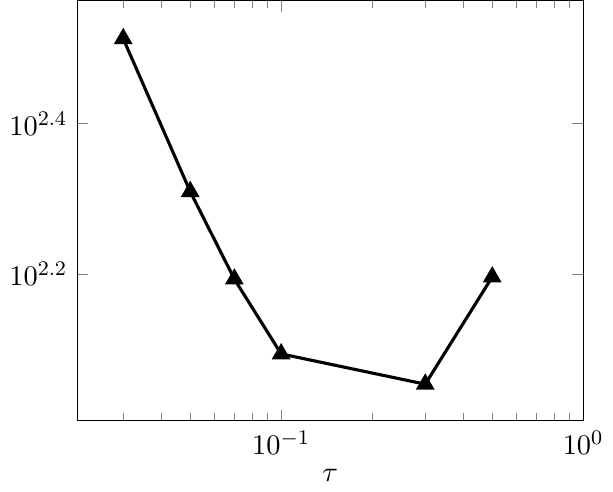}
} \qquad
\subfloat[Size of the free solution]{\label{fig:grow_norm}
\includegraphics{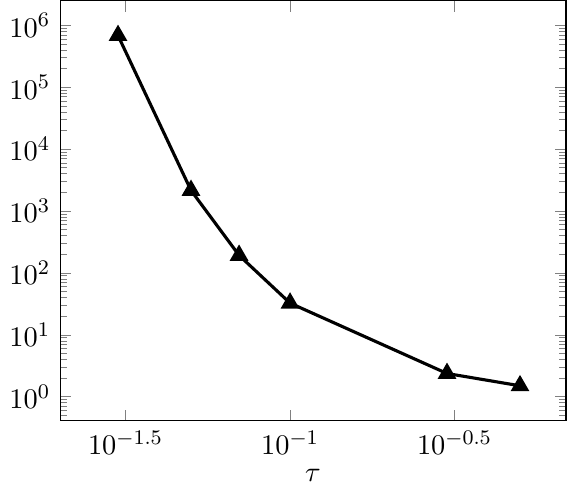}
}
\caption{Non uniformity constant with respect to $\tau$ and $\sigma$.}
\end{figure}

\subsection{Convergence of $\fint v$ to $\fint{u}$}

We conclude this section by illustrating Lemma \ref{lem:EstiMeanValueuv}. This results states that under the right configuration on the parameters of system \eqref{eq:sys_example}, the difference between the average of the component $u$ and the average of the component $v$ goes to zero as $\tau\to 0$. This result is important to establish \eqref{eq:IdentifyLimv} and from there the convergence to the nonlocal system \eqref{eq:heatSL}. To illustrate this, let us consider again the linear system \eqref{eq:sys_example} and set
\begin{equation*}
\begin{gathered}
a=-3, \quad b=2, \quad c=1, \quad d=-1, \\
u_0(x)=\sin(\pi x), \quad v_0(x)=1_{(0.2,0.7)}(x).
\end{gathered}
\end{equation*}
We take $\sigma=1/\tau$ and consider a decreasing sequence of values $\tau$. We compute numerically the difference between $\fint u_{\tau\sigma}$ and $\fint v_{\tau\sigma}$ for each pair $(\tau,\sigma)$ and compute the $L^2$-norm. In Figure \ref{fig:conv_uv} we plot the corresponding results and we observe that as $(\tau,\sigma)\to(0,+\infty)$ the norm of the difference decreases at a convergence rate of $1/2$. This experiment seems to indicate that a better convergence rate than the one specified in Lemma \ref{lem:EstiMeanValueuv} cannot be obtained.

%
%
%
%
%
%
%
%
%


\begin{figure}[htbp]
  \centering
\includegraphics{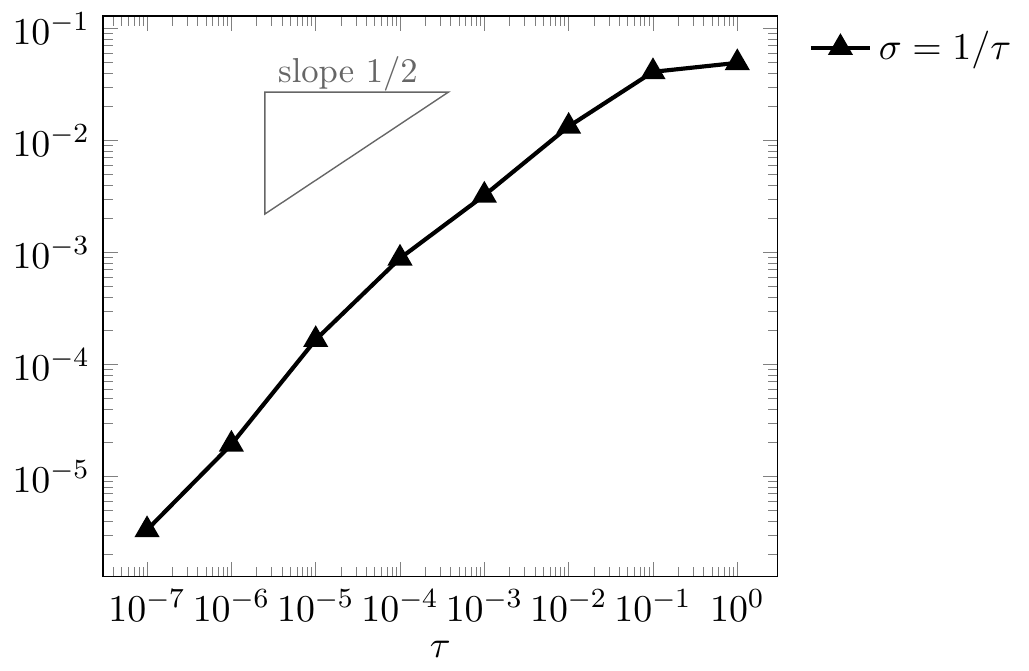}
\caption{Convergence to the average}
\label{fig:conv_uv}
\end{figure}



\section{Comments and open questions}

This section is devoted to present some additional remarks and interesting open problems concerning the controllability of nonlocal problems of the form \eqref{eq:heatSL}.

\subsection{Form of the nonlinearity}
\label{sec:formnonlinearity}
In this part, by looking carefully at the proof of \Cref{th:mainresult1} and \Cref{th:mainresult2}, we make some comments on the particular form of the nonlinearity $f$ given by \eqref{eq:Hypf}.\\
\indent First, for only proving \Cref{th:mainresult2}, we can take $g_1 \in W^{1, \infty}(\R^2)$ and $g_2 \in W^{1, \infty}(\R)$, see \Cref{sec:fixedpoint}. On the other hand, for obtaing \Cref{th:mainresult1}, we have to assume that $f$ is globally Lipschitz to pass to the limit in the system \eqref{eq:SystSL} as $(\tau, \sigma) \rightarrow (0,+\infty)$, see \Cref{sec:passlimit}. Therefore, we assume that $g_1 \in W_0^{1, \infty}(\R^2)$ and $g_2 \in W_0^{1, \infty}(\R)$. One may ask why we do not take $f$ as follows 
\begin{equation}
\label{eq:HypfBis}
\forall (u,v) \in \R^2,\ f(u,v) = a u + b v + g_1(u,v) u^2 + g_2(u,v) uv + g_3(u,v) v^2,
\end{equation}
with $g_1, g_2, g_3 \in W_0^{1, \infty}(\R^2)$. Actually, this comes from the fixed-point argument performed in \Cref{sec:fixedpoint} and the regularity estimates on the linearized system \eqref{eq:Syst_Linear_S}, see \Cref{prop:SourceTermReg}. More precisely, the second component $v$ of \eqref{eq:Syst_Linear_S} is as smooth as the first component $u$ but the maximal regularity estimates for $v$ depend on the parameter $\tau$, see \eqref{eq:EstimationLTTReg}. So, we can only use the bound in $L^2(0,T;H^2(\Omega))$ and not in $C([0,T];H^1(\Omega))$ for $v/\rho$.

\subsection{Other boundary conditions}

We may wonder to what extent our main results, i.e. \Cref{th:mainresult1} and \Cref{th:mainresult2}, can be adapted to other boundary conditions.

First, let us point out that \Cref{th:mainresult2} can be adapted to homogeneous Dirichlet boundary conditions for both the components $u$ and $v$, with slightly modifications. Actually, the crucial point is to establish the uniform null-controllability of the linearized system \eqref{eq:Syst_Linearized}. Let us remark that we can prove \Cref{prop:uniformNCLinear} for homogeneous Dirichlet boundary conditions for $(u,v)$, under a weak assumption on the coefficients $(a,b,c,d) \in \R^2 \times \R^{*} \times (-\infty, \mu_1)$ where $\mu_1$ is the first positive eigenvalue of the Dirichlet Laplacian.\\
\indent On the other hand, we do not know if \Cref{th:mainresult2} can be adapted to homogeneous Neumann boundary conditions. Indeed, in this case we do not manage to prove \Cref{prop:uniformNCLinear} because of a new difficulty appearing in the proof of the uniform global Carleman estimate, see \Cref{rmk:important}. This leads to the following open question.
\begin{op}
\label{openquestion1}
Let $(a,b,c) \in \R^2 \times \R^{*}$. The system 
\begin{equation}
\label{eq:Syst_LinearizedNeumann}
\begin{cases}
\D \partial_t u-  \Delta u = a  u +b  v +   h 1_{\omega} &\mathrm{in}\ (0,T)\times\Omega,\\
\tau \partial_t v -  \sigma \Delta v = c u + d v  &\mathrm{in}\ (0,T)\times\Omega,\\
\frac{\partial u}{\partial n} = \frac{\partial v}{\partial n}= 0,\ &\mathrm{on}\ (0,T)\times\partial\Omega,\\
(u,v)(0,\cdot)=(u_0,v_0)& \mathrm{in}\ \Omega,
\end{cases}
\end{equation}
is uniformly null-controllable with respect to the parameters $(\tau,\sigma) \rightarrow (0,+\infty)$ if and only if $d < 0$.
\end{op}
\begin{claim}
The condition $d\leq 0$ is necessary for uniform null-controllability with respect to the parameters $(\tau,\sigma) \rightarrow (0,+\infty)$ of \eqref{eq:Syst_LinearizedNeumann}.
\end{claim}
\begin{proof}
We argue by contradiction, we assume that $d>0$ and \eqref{eq:Syst_LinearizedNeumann} is uniformly null-controllable. Let us take $(u_0,v_0) \in L^2(\Omega)^2$ such that 
\[  \fint_{\Omega} c u_0 + d v_0 = 0,\ \fint_{\Omega} v_0 \neq 0.\]
This is possible because $c \neq 0$. By setting $(\alpha, \beta, \gamma)(t) = (\fint_{\Omega} u(t), \fint_{\Omega} v(t), \fint_{\Omega} h(t))$, we obtain the following ODE system from \eqref{eq:Syst_LinearizedNeumann} by integrating in the spatial variable and by using the Neumann homogeneous boundary conditions,
\begin{equation}
\label{eq:SystODE}
\begin{cases}
\dot{\alpha} = a \alpha + b \beta + \gamma &\mathrm{in}\ (0,T),\\
\tau \dot{\beta}= c \alpha + d \beta  &\mathrm{in}\ (0,T),\\
(\alpha,\beta)(0)=(\alpha_0,\beta_0).&
\end{cases}
\end{equation}
We deduce from the second equation of \eqref{eq:SystODE}
\[ \frac{d}{dt} \left( \beta(t) e^{-d t/\tau} \right) = \frac{c}{\tau} \alpha(t)  e^{-d t/\tau}.\]
Then by derivating, we have
\[ \frac{d^2}{dt^2} \left( \beta(t) e^{-t/\tau} \right) = \frac{c}{\tau} \dot{\alpha}(t)  e^{-d t/\tau} -\frac{c d }{\tau^2} \alpha(t) e^{-t/\tau} .\]
Then by using the first equation of \eqref{eq:SystODE}, we obtain
\[ \frac{d^2}{dt^2} \left( \beta(t) e^{-dt/\tau} \right) = \frac{c}{\tau} \left( a \alpha(t) + b \beta(t) + \gamma(t) \right) e^{-dt/\tau} -\frac{cd}{\tau^2} \alpha(t) e^{-dt/\tau} .\]
We integrate with respect to time to get
\[ \frac{d}{dt} \left( \beta(t) e^{-dt/\tau} \right) - \dot{\beta}(0) = \int_{0}^t  \frac{ c e^{-ds / \tau}}{\tau^2} \left( a \tau \alpha(s)  + b \tau \beta(s) + \tau \gamma(s) - d \alpha(s)\right) \d s .\]
From the assumption on the initial data, we have $\dot{\beta}(0) = \fint_{\Omega} c u_0 + d v_0 = 0$, and we integrate another time with respect to time to obtain
\[ \beta(t) e^{-dt/\tau} - \beta(0) =  \int_{0}^t \int_0^s \frac{ c e^{-dw / \tau}}{\tau^2} \left( a \tau \alpha(w)  + b \tau \beta(w) + \tau \gamma(w) - d \alpha(w)\right) \d w \d s.\]
We take $t=T$, to get the following equality because $\beta(T) = 0$,
\[ \beta(0) =  - \int_{0}^T \int_0^s \frac{ c e^{-dw / \tau}}{\tau^2} \left( a \tau \alpha(w)  + b \tau \beta(w) + \tau \gamma(w) - d \alpha(w)\right) \d w \d s.\]
By assumption on the uniform null-controllability, we know that 
\[ \gamma_{\tau}, \alpha_{\tau}, \tau \beta_{\tau}\ \text{are bounded in}\ L^2((0,T)).\]
Then, from the Cauchy-Schwarz inequality, we deduce that 
\[ |\beta(0)| \leq C \int_0^T \left(\int_0^s \frac{|c|^2 e^{-2d w/\tau}}{\tau^4} \d w\right)^{1/2} \d s \rightarrow 0\ \text{as}\ \tau \rightarrow  0\ \text{because}\ d>0.\]
Then $\beta(0) = \fint_{\Omega} v_0 = 0$, which is a contradiction.
\end{proof}
As a consequence of the previous discussion, we do not know if \Cref{th:mainresult1} can be adapted to homogeneous Neumann boundary conditions. We propose a possible approach in the next subsection.

\subsection{Shadow reaction-diffusion system}

Another way to tackle \Cref{th:mainresult2} for homogeneous boundary conditions is to try to establish directly the uniform local null-controllability of the shadow reaction-diffusion system, in the limit $\tau \rightarrow 0$.
\begin{op}
\label{openquestion2}
The PDE-ODE system 
\begin{equation}
\label{eq:Syst_Shadow}
\begin{cases}
\D \partial_t u-  \Delta u = f(u,v)+   h 1_{\omega} &\mathrm{in}\ (0,T)\times\Omega,\\
\D \tau \dot{\xi}  = \fint_{\Omega} u  - \xi  &\mathrm{in}\ (0,T),\\
\D \frac{\partial u}{\partial n} = 0,\ &\mathrm{on}\ (0,T)\times\partial\Omega,\\
u(0,\cdot) = u_0 \quad \mathrm{in}\ \Omega,\quad \xi(0) = \xi_0,
\end{cases}
\end{equation}
is uniformly locally null-controllable with respect to the parameter $\tau \rightarrow 0$.
\end{op}
We know that the result is true for Dirichlet boundary conditions by letting $\sigma \rightarrow 0$ in \Cref{th:mainresult2} and by using \cite[Theorem 4.1]{HR00}, see also \cite{rodrigues}.

In order to solve \Cref{openquestion2}, one could linearize \eqref{eq:Syst_Shadow} and prove an uniform observability estimate for the corresponding adjoint system which reads (in a general form) as
\begin{equation}\label{eq:adj_shadow}
\begin{cases}
-\partial_t\phi-\Delta \phi=a\,u+c\,\theta &\text{in } (0,T)\times\Omega,\\
-\D\tau\dot{\theta}=b\fint_{\Omega}\phi+d\,\theta &\text{in }(0,T), \\
\D \frac{\partial \phi}{\partial n} = 0,\ &\mathrm{on}\ (0,T)\times\partial\Omega,\\
\phi(T,\cdot) = \phi_T \quad \mathrm{in}\ \Omega,\quad \theta(T) = \theta_T. 
\end{cases}
\end{equation}
for some real constant coefficients $a,b,c,d$. The observability of systems like \eqref{eq:adj_shadow} has been addressed in \cite[Proposition 2.1]{HSZ18} for the case $\tau = 1$ provided $c\neq 0$. Following the proof, the main ingredients are a Carleman estimate for the first component of the system and ODE arguments for a suitable reduced system. However, so far we have encountered difficulties to follow this approach for the general case $\tau\in(0,1)$. Below, we mention them briefly. 

Arguing as in \cite{HSZ18}, we can readily obtain in a first step an inequality of the form 
\begin{equation*}
\iint_Q e^{-\frac{2C_0}{T-t}}|\phi|^2\dx\dt\leq C_1\left(\int_0^T|\theta|^2\dt+\iint_{\omega\times(0,T)}|\phi|^2\dx\dt\right)
\end{equation*}
for some positive constants $C_0$ and $C_1$ uniform with respect to $\tau$. However, it is not clear that the rest of the proof used to eliminate the integral of $\theta$ in the right-hand side can be made uniform with respect to $\tau$. Indeed, the second part of the proof focuses on studying properties of the reduced ODE system
\begin{equation}\label{eq:reduc_ODE}
\begin{cases}
-\dot{\zeta}=a\,\zeta+c\,\theta &\text{in }(0,T), \\
-\tau\dot \theta=b\,\zeta+d\,\theta &\text{in }(0,T).
\end{cases}
\end{equation}
This reduction can be easily obtained by defining $\zeta=\fint_{\Omega}\phi$ and integrating in $\Omega$ the first equation of \eqref{eq:adj_shadow}. According to \cite[Proposition 2.1]{HSZ18}, we shall look at two things: first, the regularity (in time) of the system given by
\begin{equation*}
-\tau\ddot{\theta}-(d+a)\tau\dot\theta+(bc-ad)\theta=0 \quad \text{in }(0,T)
\end{equation*}
which can be obtained by deriving with respect to time in the second equation of \eqref{eq:reduc_ODE}. However, the effects of the constant $\tau$ are not easily traceable in the arguments used \cite[Lemma 9]{HSZ18} for obtaining a good regularity result for the variable $\theta$. Secondly, a uniform observability inequality for the ODE system \eqref{eq:reduc_ODE} should be established. In the case $\tau=1$ this can be easily done by means of the classical Kalman rank criterion. Nevertheless, for systems like \eqref{eq:reduc_ODE}, the theory is far more delicate and extra assumptions are systematically used for establishing  controllability results (we refer to \cite[Section 2.6]{KKO86} for a nice compendium on controllability and observability results for systems like \eqref{eq:reduc_ODE}) and the uniformity with respect to the parameter $\tau$ is not evident. Thus, this remains as an open problem. 

\subsection{Controllability of a parabolic system with nonlocal diffusion}
In the papers \cite{FC_nonlocal1} and \cite{CFCL13}, the authors have developed theoretical and numerical results for addressing the controllability of nonlocal parabolic systems of the form
\begin{equation}\label{eq:nonlocal_diff}
\begin{cases}
\D \partial_t u-a\left(\fint_\Omega u\right)\Delta u=h1_{\omega} &\text{in } (0,T)\times\Omega, \\
u=0 &\text{on }(0,T)\times\partial\Omega, \\
u(0,\cdot)=u_0 &\text{in } \Omega,
\end{cases}
\end{equation}
where $a\in C^1(\mathbb R)$ is a function verifying
\begin{equation*}
0<m\leq a(r)\leq M \quad \forall r\in\mathbb R.
\end{equation*}

Under some assumptions on the initial data (smallness and regularity), the authors prove that system \eqref{eq:nonlocal_diff} is indeed locally null-controllable at time $T$. For this, they consider the corresponding adjoint equation linearized around the origin
\begin{equation*}
\begin{cases}
\D -\partial_t \phi-a(0)\Delta \phi=0 &\text{in } (0,T)\times\Omega, \\
\phi=0 &\text{on }(0,T)\times\partial\Omega, \\
\phi(T,\cdot)=\phi_T &\text{in } \Omega,
\end{cases}
\end{equation*}
and by means of Carleman inequalities they obtain a suitable observability inequality. Then, employing Liusternik’s inverse mapping theorem in Hilbert spaces, they are able to conclude for the original nonlinear system. 

The results in \cite{CFCL13} are also extended  to the case when a nonlinear term of the form $f(u)$ (with nice properties on the function $f$) is added to the right-hand side of \eqref{eq:nonlocal_diff}. Since one of the main ingredients of the proof  are Carleman estimates, it seems at first glance that the approach used there to treat the nonlocal diffusion is compatible with the analysis developed here for treating a nonlocal semilinear term. Moreover, the arguments developed in \cite{HR00} and employed here can be readily applied to treat equations like \eqref{eq:nonlocal_diff}. So, in this direction, a natural extension of our work is to address the controllability of the fully nonlocal parabolic equation
\begin{equation*}
\D \partial_t u-a\left(\fint_\Omega u\right)\Delta u=f\left(u,\fint_{\Omega}u\right)+h1_{\omega}
\end{equation*}
by combining the arguments in \cite{CFCL13} and the methodology developed here.

\appendix

\section{Energy estimates for the reaction-diffusion system}

In this section, we recall some classical energy estimates for the system \eqref{eq:Syst_Linearized}. More precisely, we consider for $F \in L^2((0,T)\times\Omega)$,
\begin{equation}
\label{eq:Syst_LinearizedF}
\begin{cases}
\D \partial_t u-  \Delta u = a  u +b  v +   F&\mathrm{in}\ (0,T)\times\Omega,\\
\tau \partial_t v -  \sigma \Delta v = u-v  &\mathrm{in}\ (0,T)\times\Omega,\\
\D u = \frac{\partial v}{\partial n}= 0,\ &\mathrm{on}\ (0,T)\times\partial\Omega,\\
(u,v)(0,\cdot)=(u_0,v_0)& \mathrm{in}\ \Omega.
\end{cases}
\end{equation} 
We have the following well-posedness result in $L^2$.
\begin{prop}
\label{prop:Energyestimates}
There exists a positive constant $C = C(\Omega,T)=\exp(C(\Omega)T)>0$ such that for every $(u_0,v_0) \in L^2(\Omega)^2$, $F \in L^2((0,T)\times\Omega)$, the solution $(u,v)$ to \eqref{eq:Syst_LinearizedF} satisfies
\begin{align}
&\norme{u}_{C([0,T];L^2(\Omega))} + \norme{u}_{L^2(0,T;H_0^1(\Omega))} + \norme{\partial_t u}_{L^2(0,T;H^{-1}(\Omega)')} \notag\\
& + \sqrt{\tau} \norme{v}_{C([0,T];L^2(\Omega))} + \norme{v}_{L^2(0,T;H^1(\Omega))} + \sqrt{\tau} \norme{\partial_t v}_{L^2(0,T;H^1(\Omega)')}\label{eq:EnergyEstimateL2}\\
\notag&\leq C\left( \norme{u_0}_{L^2(\Omega)} + \sqrt{\tau} \norme{v_0}_{L^2(\Omega)} + \norme{F}_{L^2((0,T)\times\Omega)}\right).
\end{align}
\end{prop}
\begin{proof}
We just give the sketch of the proof because it is standard, see \cite[Section 7.1.2]{Eva10} for the details. We only give a priori estimates. We multiply the first equation of \eqref{eq:Syst_LinearizedF} by $u$ and the second equation of \eqref{eq:Syst_LinearizedF} by $v$, then integrate in $(0,t)\times\Omega$,
\begin{align*}
\frac{1}{2} \int_{\Omega} u(t)^2 + \int_{Q_t} |\nabla u|^2 &= \frac{1}{2} \int_{\Omega} u_0^2 + \int_{Q_t} F u + \int_{Q_t} a u^2 + \int_{Q_t} b uv,\\
\frac{\tau}{2} \int_{\Omega} v(t)^2 + \sigma \int_{Q_t} |\nabla v|^2 + \int_{Q_t} v^2 &= \frac{\tau}{2} \int_{\Omega} v_0^2 + \int_{Q_t} uv.
\end{align*}
We use Youn's inequalities in the previous equations to obtain
\begin{align}
\frac{1}{2} \int_{\Omega} u(t)^2 + \int_{Q_t} |\nabla u|^2 & \leq C \left( \int_{\Omega} u_0^2 + \int_{Q_t} F^2  + \int_{Q_t} u^2 + \int_{Q_t} v^2 \right), \label{eq:EstiuSystF}\\
\frac{\tau}{2} \int_{\Omega} v(t)^2 + \sigma \int_{Q_t} |\nabla v|^2 + \frac{1}{2} \int_{Q_t} v^2 & \leq  \frac{\tau}{2} \int_{\Omega} v_0^2 +\frac{1}{2} \int_{Q_t} u^2 .\label{eq:EstivSystF}
\end{align}
We put \eqref{eq:EstivSystF} in \eqref{eq:EstiuSystF} and use Gronwall's estimate
\begin{equation}
\label{eq:EstiuSystFBis}
\int_{\Omega} u(t)^2 + \int_{Q_t} |\nabla u|^2 \leq C \left( \int_{\Omega} u_0^2 + \tau \int_{\Omega} v_0^2 + \int_{Q_t} F^2\right). 
\end{equation}
Then, we use this previous bound in \eqref{eq:EstivSystF} to get 
\begin{align}
\label{eq:EstivSystFBis}
 \tau  \int_{\Omega} v(t)^2 + \sigma \int_{Q_t} |\nabla v|^2 + \int_{Q_t} v^2 \leq C \left( \int_{\Omega} u_0^2  + \tau \int_{\Omega} v_0^2 + \int_{Q_t} F^2\right).
\end{align}
By taking the supremum for $t \in [0,T]$ in \eqref{eq:EstiuSystFBis} and \eqref{eq:EstivSystFBis}, we obtain the conclusion of the proof.
\end{proof}
We have the following maximal regularity estimate in $L^2$.
\begin{prop}
\label{prop:EnergyEstimateMaxL2}
There exists a positive constant $C = C(\Omega,T)=\exp(C(\Omega)T)>0$ such that for every $(u_0,v_0) \in H_0^1(\Omega)\times H^1(\Omega)$, $F \in L^2((0,T)\times\Omega)$, the solution $(u,v)$ to \eqref{eq:Syst_LinearizedF} satisfies
\begin{align}
&\norme{u}_{C([0,T];H_0^1(\Omega))} + \norme{u}_{L^2(0,T;H^2(\Omega))} + \norme{\partial_t u}_{L^2(0,T;L^2(\Omega))} \notag \\
& + \sqrt{\tau} \norme{v}_{C([0,T];H^1(\Omega))} + \norme{v}_{L^2(0,T;H^2(\Omega))} + \sqrt{\tau} \norme{\partial_t v}_{L^2(0,T;L^2(\Omega))} \label{eq:EnergyEstimateMaxL2}\\
&\leq C\left( \norme{u_0}_{H^1(\Omega)} + \sqrt{\tau} \norme{v_0}_{H^1(\Omega)}\right)\notag.
\end{align}
\end{prop}
\begin{proof}
It is a straightforward adaptation of the proof of \cite[Section 7.1.3, Theorem 5]{Eva10}, just by multiplying the first equation by $-\Delta u$ and the second equation by $-\Delta v$.
\end{proof}

\section{Proof of the source term method}
\label{sec:proofSTM}

In this section, we give the proof of \Cref{prop:SourceTerm}.
\begin{proof}[Proof]
\indent For $k \geq 0$, we define $T_k := T(1-q^{-k})$ where $q \in (1, \sqrt{2})$. On the one hand, let $a_0 := (u_0,\sqrt{\tau} v_0)$ and, for $k \geq 0$, we define $a_{k+1} := (u_S,\sqrt{\tau} v_S)(T_{k+1}^{-},.)$ where $(u_S,v_S)$ is the solution to
\begin{equation*}
\label{eq:SystLTT_Source}
\left\{
\begin{array}{l l}
 \partial_t u_{S} -  \Delta u_{S} = a u_S + b v_S + S&\mathrm{in}\ (T_k,T_{k+1})\times\Omega,\\
\tau \partial_t v_{S} -  \sigma \Delta v_{S} = c u_S + d v_S&\mathrm{in}\ (T_k,T_{k+1})\times\Omega,\\
u_{S}= \frac{\partial v_{S}}{\partial n }  = 0 &\mathrm{on}\ (T_k,T_{k+1})\times\partial\Omega,\\
(u,v)_{S}(T_k^{+},.)=0 &\mathrm{in}\  \Omega.
\end{array}
\right.
\end{equation*}
From classical energy estimates, see \Cref{prop:Energyestimates}, we have
\begin{equation}
\label{eq:esti_data}
\norme{a_{k+1}}_{L^{2}(\Omega)^2} \leq \norme{(u_S, \sqrt{\tau} v_S)}_{C([T_k,T_{k+1}];L^{2}(\Omega)^2)} \leq C\norme{S}_{L^{2}((T_k,T_{k+1});L^{2}(\Omega))}.
\end{equation}
On the other hand, for $k \geq 0$, we also consider the control systems
\begin{equation*}
\label{eq_SystLTT_Control}
\left\{
\begin{array}{l l}
 \partial_t u_{h} -  \Delta u_{h} = a u_h + b v_h + h 1_{\omega} &\mathrm{in}\ (T_k,T_{k+1})\times\Omega,\\
\tau \partial_t v_{h} -  \sigma \Delta v_{h} = c u_h + d v_h&\mathrm{in}\ (T_k,T_{k+1})\times\Omega,\\
u_{S}= \frac{\partial v_{h}}{\partial n }  = 0 &\mathrm{on}\ (T_k,T_{k+1})\times\partial\Omega,\\
(u_h,\sqrt{\tau} v_h)(T_k^{+},.)= a_k &\mathrm{in}\  \Omega.
\end{array}
\right.
\end{equation*}
From the null-controllability result, we deduce that we can define $h_k \in L^2((T_k, T_{k+1})\times\Omega)$ such that $(u_S,v_S)(T_{k+1}^{-},\cdot) = 0$ and thanks to the (precise) cost estimate, 
\begin{equation}
\label{eq:estiLTTControlT_k} 
\norme{h_k}_{ L^{2}((T_k,T_{k+1})\times\Omega)} \leq M e^{\frac{M}{T_{k+1}-T_{k}}} \norme{a_k}_{L^2(\Omega)^2}.
\end{equation}
In particular, for $k=0$, we have
\begin{equation*}
\label{Paper2costlinftyT0} 
\norme{h_0}_{ L^{2}((T_0,T_{1})\times\Omega)} \leq M e^{\frac{qM}{T(q-1)}} \norme{a_0}_{L^2(\Omega)^2}.
\end{equation*}
And, since $\rho_0$ is decreasing
\begin{equation}
\label{eq:estih0}
\norme{h_0/\rho_0}_{L^{2}((T_0,T_{1})\times\Omega)} \leq  \rho_0^{-1}(T_1) M e^{\frac{qM}{T(q-1)}} \norme{a_0}_{L^2(\Omega)^2}.
\end{equation}
For $k \geq 0$, since $\rho_{\mathcal{S}}$ is decreasing, combining \eqref{eq:esti_data} and \eqref{eq:estiLTTControlT_k} yields
\begin{equation}
\label{Paper2costH_k+1}
\norme{h_{k+1}}_{L^{2}((T_{k+1},T_{k+2})\times\Omega)} \leq C M e^{\frac{M}{T_{k+2}-T_{k+1}}} \rho_{\mathcal{S}}(T_k) \norme{S/\rho_{\mathcal{S}}}_{L^{2}((T_{k},T_{k+1})\times\Omega)}.
\end{equation}
In particular, by using $M e^{\frac{M}{T_{k+2}-T_{k+1}}} \rho_{\mathcal{S}}(T_k) = \rho_0(T_{k+2})$ coming from the definitions \eqref{eq:rho0} and \eqref{eq:rhoG}, we have
\begin{align}
\label{eq:estiLTTControlT_k1}
\norme{h_{k+1}}_{L^{2}((T_{k+1},T_{k+2})\times\Omega)} & \leq C  \rho_0(T_{k+2}) \norme{S/\rho_{\mathcal{S}}}_{L^{2}((T_{k},T_{k+1})\times\Omega)}.
\end{align}
Then, from \eqref{eq:estiLTTControlT_k1}, by using the fact that $\rho_0$ is decreasing, 
\begin{equation}
\label{eq:estiLTTControlT_k1rho}
\norme{h_{k+1}/\rho_0}_{L^{2}((T_{k+1},T_{k+2})\times\Omega)}  \leq C  \norme{S/\rho_{\mathcal{S}}}_{L^{2}((T_{k},T_{k+1})\times\Omega)}.
\end{equation}
As in the original proof, we can paste the controls $h_{k}$ for $k \geq 0$ together by defining
\begin{equation*}
h := \sum\limits_{k \geq 0} h_k 1_{(T_k,T_{k+1})}.
\end{equation*}
We have the estimate from \eqref{eq:estih0} and \eqref{eq:estiLTTControlT_k1rho}
\begin{equation*}
\norme{h}_{\mathcal{H}} \leq C \norme{S}_{\mathcal{S}} + C \rho_0^{-1}(T_1) M e^{\frac{qM}{T(q-1)}} \norme{a_0}_{L^2(\Omega)^2}.
\end{equation*}
The state $(u,v)$ can also be reconstructed by concatenation of $(u_S,v_S)$ + $(u_h,v_h)$, which are continuous at each junction $T_k$ thanks to the construction. Then, we estimate the state. We use the energy estimate on each time interval $(T_k, T_{k+1})$:
\begin{equation*}
\norme{(u_S, \sqrt{\tau} v_S)}_{L^{\infty}(T_k,T_{k+1};L^2(\Omega)^2)} \leq C \norme{S}_{L^2((T_k,T_{k+1})\times \Omega)},
\end{equation*}
and
\begin{equation*}
\norme{(u_h,\sqrt{\tau} v_h)}_{L^{\infty}(T_k,T_{k+1};L^2(\Omega)^2)} \leq C\left( \norme{a_k}_{L^{2}(\Omega)} +  \norme{h}_{L^{2}((T_k,T_{k+1})\times \Omega)}\right).
\end{equation*}
Proceeding similarly as for the estimate on the control, we obtain respectively
\begin{equation*}
\norme{(u_S, \sqrt{\tau} v_S)/\rho_0}_{L^{\infty}(T_k,T_{k+1};L^2(\Omega)^2)} \leq C M^{-1} \norme{S}_{\mathcal{S}},
\end{equation*}
and
\begin{equation*}
\norme{(u_h, \sqrt{\tau} v_h)/\rho_0}_{L^{\infty}(T_k,T_{k+1};L^2(\Omega)^2)} \leq C M^{-1} \norme{S}_{\mathcal{S}} + C \rho_0^{-1}(T_1) M e^{\frac{qM}{T(q-1)}} \norme{(u_0, \sqrt{\tau} v_0)}_{L^{2}(\Omega)^2}.
\end{equation*}
Therefore, for an appropriate choice of constant $C >0$, $(u,v)$ and $h$ satisfy \eqref{eq:EstimationLTT}. This concludes the proof of \Cref{prop:SourceTerm}.
\end{proof}

\renewcommand{\abstractname}{Acknowledgements}
\begin{abstract}
\end{abstract}
\vspace{-0.5cm}
Both authors benefited from the fruitful atmosphere of the conference Partial Differential Equations, Optimal Design and Numerics held at Centro de Ciencias de Benasque ``Pedro Pascual'' in August, 2019, where this work began.

\section*{Acknowledgements} 
\bibliographystyle{alpha}
\small{\bibliography{bibnonlocal}}

\end{document}